\documentclass[reqno,12pt]{amsart}%,draft
\usepackage{amsmath}
\usepackage{arydshln}
\usepackage{appendix}%
\allowdisplaybreaks[3]
\numberwithin{equation}{section}%\documentclass[reqno,12pt]{amsart}

\newtheorem{thm}{\indent Theorem}[section]
\newtheorem{cor}[thm]{\indent Corollary}
\newtheorem{lem}[thm]{\indent Lemma}
\newtheorem{prop}[thm]{\indent Proposition}
\newtheorem{dfn}{{\indent\bf Definition}}[section]
\newtheorem{rmk}{{\indent\bf Remark}}[section]

\newcommand{\ol}{\overline}
\newcommand{\ul}{\underline}

\newcommand{\ttiny}{\fontsize{5pt}{\baselineskip}\selectfont}
\newcommand{\strl}[2]{\stackrel{\mbox{\ttiny $#1$}}{#2}}

\newcommand{\td}{\tilde}

\newcommand{\fr}{\frac}

\newcommand{\edd}{\end{document}}
\newcommand{\be}{\begin{equation}}
\newcommand{\ee}{\end{equation}}

\newcommand{\lagl}{\langle}
\newcommand{\ragl}{\rangle}
\newcommand{\lmx}{\left(\begin{matrix}}
\newcommand{\rmx}{\end{matrix}\right)}
\newcommand{\ldt}{\left|\begin{matrix}}
\newcommand{\rdt}{\end{matrix}\right|}

\newcommand{\ric}{{\rm Ric}}
\newcommand{\tr}{{\rm tr\,}}
\newcommand{\vfi}{\varphi}

\newcommand{\veps}{\varepsilon}

\newcommand{\bbr}{{\mathbb R}}

\newcommand{\ba}{\begin{array}}
\newcommand{\ea}{\end{array}}
\newcommand{\nnm}{\nonumber}
\newcommand{\beal}{\begin{align}}
\newcommand{\eal}{\end{align}}
\newcommand{\bea}{\begin{eqnarray}}
\newcommand{\eea}{\end{eqnarray}}

\newcommand{\pp}[2]{\fr{\partial #1}{\partial #2}}

\newcommand{\dd}[2]{\fr{d #1}{d #2}}
\newcommand{\sch}{\mathring{h}} 
\begin{document}

\title[Blow-up of the conformal mean curvature flow]{The blow-up of the conformal mean curvature flow}

\author[X. X. Li]{Xingxiao Li$^*$}

\author[D. Zhang]{Di Zhang}

\dedicatory{}

%%%%%%%%%%%%%%% footnote %%%%%%%%%%%%%%%%
%In case \subjclass[2000] command is not effective
%(or the version of amsart.cls is old), write as follows instead:
%\renewcommand{\thefootnote}{\fnsymbol{footnote}}
%\footnote[0]{2000\textit{ Mathematics Subject Classification}.
%Primary 00; Secondary 00.}
%
\thanks{Research supported by
National Natural Science Foundation of China (No. 11671121, No. 11171091 and No. 11371018).}
%%%%%%%%%%%% Authors' addresses %%%%%%%%%%%%%
\address{% First Author
School of Mathematics and Information Sciences
\endgraf Henan Normal University \endgraf Xinxiang 453007, Henan, P.R. China
}
\email{xxl$@$henannu.edu.cn; zhangdi5727@163.com}

%\address{% Second Author} %\email{}

\begin{abstract}
In this paper, we introduce and study the conformal mean curvature flow of submanifolds of higher codimension in the Euclidean space $\bbr^n$. This kind of flow is a special case of a general modified mean curvature flow which is of various origination. As the main result, we prove a blow-up theorem concluding that, under the conformal mean curvature flow in $\bbr^n$, the maximum of the square norm of the second fundamental form of any compact submanifold tends to infinity in finite time. Furthermore, by using the idea of Andrews and Baker for studying the mean curvature flow of submanifolds in the Euclidean space, we also derive some more evolution formulas and inequalities which we believe to be useful in our further study of conformal mean curvature flow. Presently, these computations together with our main theorem are applied to provide a direct proof of a convergence theorem concluding that the external conformal forced mean curvature flow of a compact submanifold in $\bbr^n$ with the same pinched condition as Andrews-Baker's will be convergent to a round point in finite time.
\end{abstract}

\maketitle

\tableofcontents

\section{Introduction}

As is known, the mean curvature flow (MCF) was proposed in 1956 by W. Mullins to describe the formation of grain boundaries in annealing metals. Brakke (\cite{bra}) introduced the motion of submanifolds by MCF in arbitrary codimension and constructed a generalized varifold solution for all time. Since then there have been fruitful interesting results on MCF up to now, in particular, for hypersurfaces in Euclidean space. For example, Huisken (\cite{hui84}) showed that any compact and uniformly convex hypersurface in the Euclidean space is convergent under the MCF to a round point in a finite time and, in the case of higher codimension, Andrews and Baker proved (\cite{a-b} or \cite{b}; see Theorem \ref{thmab} below) that if the initial submanifold is compact and its second fundamental form satisfies a suitable pinching condition, then the corresponding MCF in the Euclidean space must be convergent to a round point in finite time. The latter theorem was
later generalized to MCFs in both spherical and hyperbolic space forms, see Baker (\cite{b}) and Liu-Xu-Ye-Zhao (\cite{l--z11} and \cite{l--z}). For other progresses on the MCFs, we refer the readers to the references \cite{andr}, \cite{l--z12}, \cite{smo11} and \cite{w} etc.

In this paper we aim to study some more general flow that, in a direction, generalizes the usual MCF. The motivation of our consideration is as follows:

Let $M$ be a compact manifold of dimension $m$, and $(N,\ol g)$ a Riemannian manifold of dimension $n:=m+p$ with $p\geq 1$. Denote
by ${\mathcal F}(M,N)$ the set of all smooth immersions of $M$ into $N$. For a given $F_0\in {\mathcal F}(M,N)$, one may consider the following modified mean curvature flow with an external force $\ol W\in\Gamma(TN)$:
\be\label{0}
\begin{cases}\pp{F}{t}=a(F,t)H_F+\phi(t)\ol W\circ F,\\
F(\cdot,0)\equiv F_0
\end{cases}
\ee
where $F_t:=F(\cdot,t)\in {\mathcal F}(M,N)$, $a\in C^\infty(N\times[0,T_0))$ is a positive smooth function for some large $T_0>0$, $H:=H_F$ is the mean curvature of $F_t:M\to N$, and $\phi\in C^\infty[0,T_0)$.

\begin{rmk} Since the tangential component $(\ol W\circ F)^\top$ of $\ol W\circ F$ does not essentially affect the behavior of the evolving of submanifolds, the curvature flow \eqref{0} can be equivalent to the following flow of the normal version:
\be\label{1}
\begin{cases}
\pp{F}{t}=a(F,t)H_F+\phi(t)(\ol W\circ F)^\bot,\\
F(\cdot,0)\equiv F_0.
\end{cases}
\ee
\end{rmk}

Note that special cases of \eqref{0} or \eqref{1} are well-known,  among which we list a few:

(1) The most important case is the mean curvature flow which corresponds to
$a\equiv 1$, and $\phi\equiv 0$ or $\ol W\equiv 0$:
\be\label{mcf}
\pp{F}{t}=H_F,\quad
F(\cdot,0)\equiv F_0.
\ee

(2) The externally forced mean curvature flow ($a\equiv 1$). This case has also been studied by many authors in recent years from different point of views. For example, the mean curvature flow with density (see \cite{b-r10} and \cite{b-r14} for Gauss mean curvature flow in real space forms);  that in the Euclidean space $\bbr^n$ with the external force in the direction of position vector (\cite{g-l-w}, \cite{s-s}); some more general flows in $\bbr^n$ when $\phi\equiv 1$ and $\ol W=\bar\nabla\psi$ for certain smooth functions $\psi\in C^\infty(\bbr^n)$ (\cite{l-j07}, \cite{l-j08} and \cite{l-s}), which are related to the study of the Ginzburg-Landau vortex (\cite{j-x03} and \cite{j-l06}).

(3) Let $\rho>0$ be a smooth function on $N$ and $\td g=\rho^2\ol g$. Then we can consider the MCF in the new Riemannian manifold $(N,\td g)$:
\be\label{2}
\pp{}{t}F=\td H_F
\ee
where $\td H_F$ is the mean curvature of the same immersion $F_t:M\to N$ but with respect to the conformal metric $\td g$. By a direct computation, one easily find that in terms of the mean curvature $H$ of the original submanifolds of $(N,\ol g)$, \eqref{2} is changed into
\be\label{3}
\pp{F}{t}=\rho^{-2}(F)\left(H_F+m(\bar\nabla_{\ol g}\log\rho)^\bot\circ F\right).
\ee
This is of special significance because, given an arbitrarily Riemannian manifolds $(N,\td g)$, the MCF in $(N,\td g)$ may be alternately studied by choosing a possibly simpler or standard metric $\ol g$ in the conformal class, or vice versa. Take, say, $N=\bbr^n$.

Presently, we are mainly interested in a special case of \eqref{0} or \eqref{1} when $\phi\equiv0$ or $\ol W\equiv 0$, that is, we are to consider the following flow of submanifolds:
\be\label{4}
\begin{cases}
\pp{F}{t}=a(F,t)H_F,\\
F(\cdot,0)\equiv F_0
\end{cases}
\ee
where $a_t=a(\cdot,t)$ is a fixed family of positive smooth functions on $N$, and $F_0:M\to N$ is a given immersion. By using the known trick of De Turck, it is not hard to show that \eqref{4} has a short-time existence of solution for each $F_0$ (see Theorem \ref{exiuni} in section \ref{s2}). Apparently, a flow of the form \eqref{4} can be viewed as the flow of conformal maps driven by the normal tension: the normal part of the tension field $\tau$ of the comformal map $F_t:(M,a^{-2}(F(\cdot,t),t)g)\to (N,\ol g)$ is exactly $a(F,t)H_F$, where $g=F^*\ol g$ is the induced metric via $F$.

On the other hand, if $\ol W$ is a conformal vector field on $(N,\ol g)$ with the one-parametric transformations $-\bar\vfi_s$, then we are able to prove (see Theorem \ref{thm3.1} in Section\ref{s3}) that, up to some diffeomorphisms on $M$, the following mean curvature flow
\be\label{5}
\pp{F}{t}=H_F+\phi(t)\ol W\circ F,\quad F(\cdot,0)=F_0(\cdot)
\ee
or
\be\label{6}
\pp{F}{t}=H_F+\phi(t)(\ol W\circ F)^\bot,\quad F(\cdot,0)=F_0(\cdot)
\ee
with an external force $\ol W$ is equivalent to a special kind of flow in the form
\be\label{3-1}
\pp{F}{t}=\rho^2(F,t)H_F+m\bar\vfi_{\bar t*}(\bar\nabla\log\rho)^\bot
\ee
where $\rho=\rho(p,t)$ is given by $(\bar\vfi_{\bar t})^*\ol g=\rho^2\ol g$ and $\bar t=\int^t_0\phi(s)ds$. In particular, if $\rho$ only depends on the parameter $t$, then \eqref{3-1} will assume a special case of the form \eqref{4}.

Note that, when the external force $\ol W$ is a closed conformal vector field and $\phi\equiv 1$, the corresponding flow \eqref{6} has been systematically studied in \cite{mntl} and, very recently, the general MCF solitons in the presence of conformal vector fields are studied in \cite{a--r}. However, we also know from \cite{mntl} that the existence of a closed conformal vector field $\ol W$ is rather restrictive for the ambient Riemannian manifold $(N,\ol g)$.

The above discussion naturally leads to the following definition:
\begin{dfn}
The modified mean curvature flows \eqref{4} is called the {\em conformal mean curvature flow (CMCF)}.
\end{dfn}

To study the conformal mean curvature flow, it seems natural for us to first consider the simplest but the most important case that the target $(N,\ol g)$ is taken to be the Euclidean space $\bbr^n$ with $\ol g$ the standard flat metric, where we shall use $(y^A)$ to denote the standard orthogonal coordinates. Thus a map $F\in{\mathcal F}(M):\equiv {\mathcal F}(M,\bbr^n)$ can be expressed by its component functions $F^1,\cdots, F^n$, that is, $F=(F^A):=(F^1,\cdots,F^n)$.

In this paper, therefore, we mainly study the CMCF \eqref{4} with $N\equiv\bbr^n$ and prove a blow-up theorem as follows:
\begin{thm}[see Theorem \ref{thm5.1}]\label{main1}
Let $M$ be a compact manifold of dimension $m\geq 2$ and $a\in C^\infty(\bbr^n\times[0,T_0))$ be a positive function. Then for any given $F_0\in{\mathcal F}(M)$, there exists a maximal and finite $T>0$ such that the CMCF \eqref{4} has a unique maximal solution $F:M\times[0,T)\to\bbr^n$ which blows up at the time $T$ in the sense that $\lim\limits_{t\to T}\max_M|h|^2=+\infty$, where $h\equiv h_t$ is the second fundamental form of the immersion $F_t:M\to\bbr^n$.
\end{thm}

We also follow the idea of \cite{a-b} to derive some general formulas and inequalities for the flow \eqref{4} in case that $N=\bbr^n$, which we put in the appendix of this paper. We reasonably believe that these computations will be useful in the further study of \eqref{4} in the Euclidean space. Currently, as a direct application, we alternatively give a direct proof of the following theorem (in Section \ref{s7}) which generalizes one of the main theorems of \cite{b-r14}:

\begin{thm}[see Theorem \ref{thm3.2}]\label{main2}
Let $M$ be as in Theorem \ref{main1}. Suppose that the initial immersion $F_0\in{\mathcal F}(M)$ satisfies the following two conditions:

(1) the mean curvature $H$ does not vanish everywhere;

(2) the square of the norm of the second fundamental form $|h|^2\leq c|H|^2$ for some constant $c$ satisfying
\be\label{c}c\leq\fr4{3m}\text{ in case }2\leq m\leq 4;\quad c\leq\fr1{m-1}\text{ in case } m\geq 5.\ee
Then the mean curvature flow \eqref{5} or \eqref{6} with an external conformal force has a unique smooth solution $F:M\times[0,T)\to\bbr^{m+p}$ on a finite maximal time interval, and $F_t(M)$ converges uniformly to a round point in $\bbr^{m+p}$.
\end{thm}

We should remark that, to our point of view, the main theorem of \cite{a-b} is among the first and the most important generalizations of the famous convergence theorem for convex hypersurfaces (\cite{hui84}) by Huisken to the case of higher codimension. Here we would like to restate the theorem of Andrews and Bakes as follows:

\begin{thm}\label{thmab}
Let $M$ be as in Theorem \ref{main1}. If for the initial immersion $F_0\in{\mathcal F}(M)$, the mean curvature $H\neq 0$ everywhere and $|h|^2\leq c|H|^2$ for some constant $c$ satisfies \eqref{c}, then the mean curvature flow \eqref{mcf} has a unique smooth solution on a finite maximal time interval, and it converges uniformly to a round point.
\end{thm}

Obviously, Theorem \ref{main2} generalizes Theorem \ref{thmab} in a way, which we have shown to be a direct corollary of Theorem \ref{thm3.1} and Theorem \ref{thmab} together with a well-known Liouville theorem for conformal transformations on the Euclidean space (see Section \ref{s3}).

\section{Short time existence of the conformal mean curvature flow}\label{s2}

Same as the standard mean curvature flow, Equation \eqref{4} is generally a degenerate parabolic partial equation. To be able to use the standard theory of parabolic equations, it is convenient to appeal the De Turck trick (see, for example,\cite{dtu}, \cite{b} etc).

First we give some preparation on notations. For any given Riemannian metric $g$ on $M$, denote by $\nabla^g$ the Levi-Civita connection of $g$. By fixing arbitrarily a Riemannian metric $\td g$ we define a metric-dependent vector field $W=W(g)$ on $M$ such that $W(g)=\tr_g(\nabla^g-\nabla^{\td g})$. In particular, $W$ restricts to an immersion-dependent vector field $W(F)\equiv W(g_F)$ for $F\in {\mathcal F}(M,N)$ where $g:=g_F$ is the induced metric via $F$ of the metric $\ol g$ on the target $N$.

Now, similar to the case of mean curvature flow,  we introduce the following {\em De Turck mean curvature flow}:
\be\label{dtmcf}
\pp{\hat F}{t}=a(\hat F,t)(H(\hat F)+\hat F_*W(\hat F)).
\ee
Under local coordinate systems $(x^i)$ on $M$ and $(y^A)$ on $N$, respectively, we write
\begin{align*}
e_i=&\pp{}{x^i},\quad F_i\equiv F_{t*}(e_i)=\sum_iF^A_i\pp{}{y^A},\\
W(g)=&W^ke_k=g^{ij}(\Gamma^k_{ij}-\td\Gamma^k_{ij})e_k,
\end{align*}
where $\Gamma^k_{ij}$ and $\td\Gamma^k_{ij}$ are the Christoffel symbols for $g$ and $\td g$, respectively.
Then the two flows \eqref{4} and \eqref{dtmcf} have the following local representations respectively:
\be\label{cfmcf1}
\pp{F^A}{t}=a(F,t)g_F^{ij}F^A_{,ij}
\ee
\be\label{dtmcf1}
\pp{\hat F^A}{t}=a(\hat F,t)(g_{\hat F} ^{ij}\hat F^A_{,ij}+\hat F^A_kW^k(\hat F))\equiv a(\hat F,t)g_{\hat F} ^{ij}\hat F^A_{;ij},
\ee
where the subscript ``$,$'' denotes the covariant derivatives w.r.t the time-dependent metric $g_F$ and $g_{\hat F}$ accordingly, while the subscript $``;$'' denotes the covariant derivative w.r.t the fixed metric $\td g$. From \eqref{dtmcf1} it is clearly seen that \eqref{dtmcf} is a (nondegenerate) parabolic equation and thus it has a short-time existent solution $\hat F=\hat F(x,t)$, $t\in [0,T)$, according to the standard theory of parabolic equations. So we have a well-defined time-dependent vector field $W=W(\hat F)$ on $M$ for all $t\in [0,T)$.

Now we recall a known existence result as follows:

\begin{lem}[see for example \cite{c-k}, p.82, Lemma 3.15]\label{tmdptfl} If $\{X_t:0 < t < T \leq\infty\}$ is a continuous time-dependent
family of vector fields on a compact manifold $M$, then there exists uniquely a one-parameter
family of diffeomorphisms
$$\{\vfi_t:M\to M; 0\leq t < T \leq\infty\}$$
defined by $\vfi:M\times[0,T)\to M$
on the same time interval such that
$$\pp{\vfi_t}{t} (x)\equiv\vfi_*\left(\pp{}{t}\right) = X_t(\vfi_t(x)),\quad
\vfi_0(x) = x.
$$
for all $x\in M$ and $t\in [0, T)$.
\end{lem}

Substituting $X$ with $-a(\hat F,t)W(\hat F)$ we obtain a family of diffeomorphisms $\vfi_t\equiv\vfi(\cdot,t)$, $0\leq t<T$, on $M$. Write $\td x^i_t=\vfi^i_t(x)$ ($0\leq t<T$). Then $(\td x^i_t)$ is a family of local coordinate systems with the parameter $t\in [0,T)$. Define a family of immersions $F_t(x):=\hat F(\vfi_t(x),t)$ and it is easy to see that $g(\hat F_t)=g(F_t)$.

Given a Riemannian metric $g$ on $M$ and an $F\in\mathcal{F}(M,N)$, we always use $\lagl\cdot,\cdot\ragl_g$ to denote the induced inner product on the vector bundle $T^r_s(M)\otimes F^*TN$ by the metrics $g$ and $\ol g$, where $T^r_s(M)$ is the $(r,s)$-tensor bundle on $M$. In particular, we shall omit the subscript $g$ in $\lagl\cdot,\cdot\ragl_g$ when $g$ is the induced metric by $F$. From this we compute
\begin{align*}
\pp{F}{t}\equiv& F_*\left(\pp{}{t}\right)=\left.\hat F_*\left(\pp{}{t}\right)\right|_{(\vfi(x,t),t)}+(\hat F_t)_*\circ\vfi_*\left(\pp{}{t}\right)\\
=&a(\hat F(\vfi(x,t),t),t)g_{\hat F} ^{ij}\hat F_{;ij}-a(\hat F(\vfi(x,t),t),t) (\hat F_t)_*(W(\hat F(\vfi(x,t),t)))\\
=&a(\hat F(\vfi(x,t),t),t)g_{\hat F} ^{ij}(\hat F_{;ij}-\hat F_k(\hat\Gamma^k_{ij}-\td\Gamma^k_{ij}))\\
=&a(\hat F(\vfi(x,t),t),t)g_{\hat F} ^{ij} \hat F_{,ij}\\
=&a(F(x,t),t)g_F^{ij}F_{,ij}=a(F,t)H(F).
\end{align*}
This shows that $F(x,t)$ is a solution of the flow \eqref{4}.

Conversely, for a given solution $F=F(x,t)$ of the flow \eqref{4}, we can similarly find another time-dependent vector field $\hat W(F)$ with the corresponding one-parameter transformations $\hat\vfi_t$. Then we obtain a family of immersions $\hat F(x,t)=F(\hat\vfi(x,t),t)$ which solve the De Turck mean curvature flow \eqref{dtmcf}.

The above argument gives the following existence and uniqueness theorem:

\begin{thm}\label{exiuni}
For any $F_0\in {\mathcal F}(M,N)$, there exists a maximal $T$: $0<T\leq +\infty$ with a unique smooth solution $F:M\times[0,T)\to N$ to the CMCF \eqref{4}.
\end{thm}

\section{The mean curvature flow with a conformal external force}\label{s3}

In this section, we aim to deal with the modified mean curvature flow \eqref{5} or \eqref{6} with a conformal external force. Then, as a direct application, we shall give a convergence theorem for the mean curvature flow with a conformal external force in the Euclidean space (see Theorem \ref{thm3.2} below).

For convenience, define
\be\bar t\equiv \bar t(t):=\int^t_0\phi(s)ds.\ee

Firstly, we are to prove the following theorem:

\begin{thm}\label{thm3.1} Let $\ol W$ be a conformal vector field on the Riemannian manifold $(N,\ol g)$ and $\bar\vfi:N\times(-\veps,\veps)\to N$ be the one-parameter
family of diffeomorphisms induced by $-\ol W$. Suppose that $F:M\times [0,T)\to N$ is a solution of \eqref{5} or \eqref{6} for some $T>0$, and $\bar t(t)\in (-\veps,\veps)$ for all $t\in [0,T)$. Then, up to a diffeomorphism on $M$, the map $\hat F:M\times [0,T)\to N$ defined by
\be\label{3.2}
\hat F(p,t):=\bar\vfi(F(p,t),\bar t(t)),\quad \forall (p,t)\in M\times [0,T)
\ee
is a solution of the curvature flow \eqref{3-1} with an external force, where $\rho>0$ is given by $\bar\vfi^*_{\bar t}\ol g=\rho^2\ol g$. In particular, $\hat F$ and $F$ have the same initial submanifold $\hat F_0=F_0$ and, for each moment of time $t$, the images of the corresponding immersions $F$ and $\hat F$ are globally conformal to each other.
\end{thm}

\begin{proof} We prove this theorem by direct computations as follows (taking \eqref{6} as the example): For any $p\in M$ and $t\in [0,T)$, we have
\be\label{barfi}
\pp{\bar\vfi(F(p,t),s)}{s}=-\ol W(\bar\vfi(F(p,t),s)).
\ee
Therefore
\begin{align}
\pp{\hat F(p,t)}{t}=&\pp{\bar\vfi(F(p,t),\bar t(t))}{\bar t}\dd{\bar t}{t}+\bar\vfi_{\bar t*F(p,t)}\left(\pp{F(p,t)}{t}\right)\nnm\\
=&-\phi(t)\ol W(\bar\vfi(F(p,t),\bar t(t)))+\bar\vfi_{\bar t*F(p,t)}(H+\phi(t)\ol W^\bot)\nnm\\
=&-\phi(t)\ol W(\hat F(p,t))+\phi(t)\bar\vfi_{\bar t*F(p,t)}(\ol W^\bot(F(p,t)))+\bar\vfi_{\bar t*F(p,t)}(H).\label{eq3.2}
\end{align}
Moreover, by \eqref{barfi}, we also find
\begin{align}
\bar\vfi_{\bar t*F(p,t)}(\ol W(F(p,t)))=&-\bar\vfi_{\bar t*F(p,t)}\left(\pp{}{s}\Big|_{s=0}\bar\vfi(F(p,t),s)\right)=-\pp{}{s}\Big|_{s=0}\bar\vfi(F(p,t),\bar t(t)+s)\nnm\\
=&\ol W(\bar\vfi(F(p,t),\bar t(t)))=\ol W(\hat F(p,t)).\label{3.3}
\end{align}
Since $\bar\vfi_{\bar t}$ is conformal on $N$, it sends a normal (resp. tangent) vector to a normal (resp. tangent) vector. It then follows from \eqref{3.3} that
$$
\bar\vfi_{\bar t*F(p,t)}(\ol W^\top(F(p,t)))=\ol W^\top(\hat F(p,t)),\quad \bar\vfi_{\bar t*F(p,t)}(\ol W^\bot(F(p,t)))=\ol W^\bot(\hat F(p,t)).
$$
Inserting the second formula into \eqref{eq3.2}, we obtain that
\begin{align}
\pp{\hat F(p,t)}{t}=&-\phi(t)\ol W(\hat F(p,t))+\phi(t)\ol W^\bot(\hat F(p,t))+\bar\vfi_{\bar t*F(p,t)}(H)\nnm\\
=&-\phi(t)\ol W^\top(\hat F(p,t))+\bar\vfi_{\bar t*F(p,t)}(H).\label{eq3.3}
\end{align}

On the other hand, by fixing a $t\in [0,T)$ we get a fixed $\bar t=\bar t(t)$. Denote by $h^{\bar\vfi_{\bar t}}$ the second fundamental form of $\bar\vfi_{\bar t}$ as a smooth map. Then
\be\label{eq3.4}
\bar{D}_{\bar X}\bar\vfi_{\bar t*}(\bar Y)=\bar\vfi_{\bar t*}(\bar{D}_{\bar X}\bar Y)+h^{\bar\vfi_{\bar t}}(\bar X,\bar Y),\quad \forall \bar X,\bar Y\in \Gamma(TN).
\ee
Since $(\bar\vfi_{\bar t})^*_{q}\bar{g}_{\bar\vfi_{\bar t}(q)}=\rho^{2}\bar{g}_q$ for any $q\in N$, we claim that
\be\label{eq3.5}
h^{\bar\vfi_{\bar t}}(\bar X,\bar Y)=\bar X(\log\rho)\bar\vfi_{\bar t*}(\bar Y)+\bar Y(\log\rho)\bar\vfi_{\bar t*}(\bar X)-\lagl\bar X,\bar Y\ragl\bar\vfi_{\bar t*}(\bar\nabla\log\rho).
\ee

In fact, by the Koszul formula for Riemannian connections, for an arbitrary $\bar Z\in\Gamma(TN)$,
\begin{align*}
2\lagl\bar{D}_{\bar X}&\bar\vfi_{\bar t*}(\bar Y),\bar\vfi_{\bar t*}(\bar Z)\ragl
=\bar\vfi_{\bar t*}(\bar X)\lagl\bar\vfi_{\bar t*}(\bar Y),\bar\vfi_{\bar t*}(\bar Z)\ragl +\bar\vfi_{\bar t*}(\bar Y)\lagl\bar\vfi_{\bar t*}(\bar Z),\bar\vfi_{\bar t*}(\bar X)\ragl \\
&-\bar\vfi_{\bar t*}(\bar Z)\lagl\bar\vfi_{\bar t*}(\bar X),\bar\vfi_{\bar t*}(\bar Y)\ragl
+\lagl[\bar\vfi_{\bar t*}(\bar X),\bar\vfi_{\bar t*}(\bar Y)],\bar\vfi_{\bar t*}(\bar Z)\ragl \\
&+\lagl[\bar\vfi_{\bar t*}(\bar Z),\bar\vfi_{\bar t*}(\bar X)],\bar\vfi_{\bar t*}(\bar Y)\ragl -\lagl[\bar\vfi_{\bar t*}(\bar Y),\bar\vfi_{\bar t*}(\bar Z)],\bar\vfi_{\bar t*}(\bar X)\ragl \\
=&\bar\vfi_{\bar t*}(\bar X)(\rho^2\lagl \bar Y,\bar Z\ragl )+\bar\vfi_{\bar t*}(\bar Y)(\rho^2\lagl \bar Z,\bar X\ragl )-\bar\vfi_{\bar t*}(\bar Z)(\rho^2\lagl \bar X,\bar Y\ragl )\\
&+\rho^2(\lagl[\bar X,\bar Y],\bar Z\ragl +\lagl[\bar Z,\bar X],\bar Y\ragl -\lagl[\bar Y,\bar Z],\bar X\ragl )\\
=&2\rho(\bar X(\rho)\lagl \bar Y,\bar Z\ragl
+\bar Y(\rho)\lagl \bar Z,\bar X\ragl -\bar Z(\rho)\lagl \bar X,\bar Y\ragl )\\
&+\rho^2(\bar X\lagl \bar Y,\bar Z\ragl +\bar Y\lagl \bar Z,\bar X\ragl
-\bar Z\lagl \bar X,\bar Y\ragl \\
&+\lagl[\bar X,\bar Y],\bar Z\ragl +\lagl[\bar Z,\bar X],\bar Y\ragl -\lagl[\bar Y,\bar Z],\bar X\ragl )\\
=&2\big(\bar X(\log\rho)\lagl\bar\vfi_{\bar t*}(\bar Y),\bar\vfi_{\bar t*}(\bar Z)\ragl +\bar Y(\log\rho)\lagl\bar\vfi_{\bar t*}(\bar X),\bar\vfi_{\bar t*}(\bar Z)\ragl \\
&-\lagl \bar X,\bar Y\ragl \lagl\bar\vfi_{\bar t*}(\bar\nabla\log\rho),\bar\vfi_{\bar t*}(\bar Z)\ragl \big)
+2\rho^2\lagl\bar{D}_{\bar X}\bar Y,\bar Z\ragl \\
=&2\big(\bar X(\log\rho)\lagl\bar\vfi_{\bar t*}(\bar Y),\bar\vfi_{\bar t*}(\bar Z)\ragl +\bar Y(\log\rho)\lagl\bar\vfi_{\bar t*}(\bar X),\bar\vfi_{\bar t*}(\bar Z)\ragl \\
&-\lagl \bar X,\bar Y\ragl \lagl\bar\vfi_{\bar t*}(\bar\nabla\log\rho),\bar\vfi_{\bar t*}(\bar Z)\ragl +\lagl\bar\vfi_{\bar t*}(\bar{D}_{\bar X}\bar Y),\bar\vfi_{\bar t*}(\bar Z)\ragl \big)\\
=&2\lagl \bar X(\log\rho)\bar\vfi_{\bar t*}(\bar Y)+\bar Y(\log\rho)\bar\vfi_{\bar t*}(\bar X)\\
&-\lagl \bar X,\bar Y\ragl \bar\vfi_{\bar t*}(\bar\nabla\log\rho)+\bar\vfi_{\bar t*}(\bar{D}_{\bar X}\bar Y),\bar\vfi_{\bar t*}(\bar Z)\ragl .
\end{align*}
So it holds that
$$
\bar{D}_{\bar X}\bar\vfi_{\bar t*}(\bar Y)=\bar\vfi_{\bar t*}(\bar{D}_{\bar X}\bar Y)+\bar X(\log\rho)\bar\vfi_{\bar t*}(\bar Y)+\bar Y(\log\rho)\bar\vfi_{\bar t*}(\bar X)-\lagl\bar X,\bar Y\ragl\bar\vfi_{\bar t*}(\bar\nabla\log\rho),
$$
which with \eqref{eq3.4} proves the desired formula \eqref{eq3.5}.

Now for any $X,Y\in\Gamma(M)$, by letting $\bar X=F_{t*}(X)$ and $\bar Y=F_{t*}(Y)$, we get
\begin{align*}
\bar{D}_{X}\hat F_{*}(Y)
=&\bar{D}_{X}\bar\vfi_{\bar t*}(F_{*}(Y))
=\bar\vfi_{\bar t*}(\bar{D}_{X}F_{*}(Y))+X(\log\rho)\bar\vfi_{\bar t*}(F_{*}(Y))\\
&+Y(\log\rho)\bar\vfi_{\bar t*}(F_{*}(X))-\lagl F_{*}(X),F_{*}(Y)\ragl\bar\vfi_{\bar t*}\big(\bar\nabla\log\rho\big)\\
=&\bar\vfi_{\bar t*}\big(F_{*}(D_{X}Y)+h(X,Y)\big)+X(\log\rho)\hat F_{*}(Y)\\
&+Y(\log\rho)\hat F_{*}(X)-\lagl X,Y\ragl \big(\hat F_{*}(\nabla\log\rho)+\bar\vfi_{\bar t*}(\bar\nabla\log\rho)^\bot\big)
\end{align*}
where $h$ is the second fundamental form of $F_t$. Consequently, if $\hat D$ is the Levi-Civita connection of the induced metric $\hat g\equiv (\hat F_t)^*\ol g$ on $M$, then the second fundamental forms $\hat h$ of the immersions $\hat F_t$ is given by
\begin{align}
\hat h(X,Y)=&\bar{D}_{X}\hat F_{*}(Y)-\hat F_{*}(\hat{D}_{X}Y)\nnm\\
=&\bar\vfi_{\bar t*}\big(F_{*}(D_{X}Y)+h(X,Y)\big)+X(\log\rho)\hat F_{*}(Y)+Y(\log\rho)\hat F_{*}(X)\nnm\\
&-\lagl X,Y\ragl \hat F_{*}(\nabla\log\rho)-\hat F_{*}(\hat{D}_{X}Y)-\lagl X,Y\ragl \bar\vfi_{\bar t*}(\bar\nabla\log\rho)^\bot\nnm\\
=&\bar\vfi_{\bar t*}(h(X,Y))+\bar\vfi_{\bar t*}\big(F_{*}(D_{X}Y)\big)-\hat F_{*}(\hat{D}_{X}Y)\nnm\\
&+X(\log\rho)\hat F_{*}(Y)+Y(\log\rho)\hat F_{*}(X)\nnm\\
&-\lagl X,Y\ragl \hat F_{*}(\nabla\log\rho)-\lagl X,Y\ragl \bar\vfi_{\bar t*}(\bar\nabla\log\rho)^\bot \nnm\\
=&\bar\vfi_{\bar t*}(h(X,Y))-\lagl X,Y\ragl \bar\vfi_{\bar t*}(\bar\nabla\log\rho)^\bot\nnm\\
&+\hat F_{*}\big(D_{X}Y-\hat{D}_{X}Y+X(\log\rho)Y+Y(\log\rho)X -\lagl X,Y\ragl \nabla\log\rho\big)\nnm\\
=&\bar\vfi_{\bar t*}(h(X,Y))-\lagl X,Y\ragl \bar\vfi_{\bar t*}(\bar\nabla\log\rho)^\bot.
\end{align}
Thus the mean curvature $\hat H$ of $\hat F_t$ and the curvature $H$ of $F_t$ are related by
\begin{align*}
\hat H=&\tr_{\hat g}\hat h=\rho^{-2}\tr_g\big(\bar\vfi_{\bar t*}(h)-g\bar\vfi_{\bar t*}(\bar\nabla\log\rho)^\bot\big)\\
=&\rho^{-2}(\bar\vfi_{\bar t*}(\tr_gh)-m\bar\vfi_{\bar t*}(\bar\nabla\log\rho)^\bot)\\
=&\rho^{-2}(\bar\vfi_{\bar t*}(H)-m\bar\vfi_{\bar t*}(\bar\nabla\log\rho)^\bot).
\end{align*}
It then follows from \eqref{eq3.3} that
\begin{align*}
\pp{\hat F}{t}=&-\phi(t)\ol W^\top(\hat F(p,t))+\rho^2\hat H+m\bar\vfi_{\bar t*}(\bar\nabla\log\rho)^\bot,\\
&\left(\pp{\hat F}{t}\right)^\bot=\rho^2\hat H+m\bar\vfi_{\bar t*}(\bar\nabla\log\rho)^\bot.
\end{align*}
Then Theorem \ref{thm3.1} is proved.
\end{proof}

Next, as a direct application of Theorem \ref{thm3.1} and Theorem \ref{thmab}, we are able to prove the following theorem, generalizing one of the main results in \cite{b-r14}:

\begin{thm}\label{thm3.2}
Let $M$ and the initial immersion $F_0\in{\mathcal F}(M)$ be as in Theorem \ref{main2}. Then the mean curvature flow \eqref{5} or \eqref{6} with an external conformal force $\ol W$
has a unique smooth solution $F:M\times[0,T)\to\bbr^{m+p}$ on a finite maximal time interval $[0,T)$, and $F_t(M)$ converges uniformly to a round point in $\bbr^{m+p}$ as $t\to T$.
\end{thm}

\begin{proof}
First of all, we make use of a theorem of Liouville (see \cite{m}, Appendix 6; also \cite{har}) to obtain that,  for any conformal vector field $\ol W$ on the Euclidean space $\bbr^n$, the conformal transformations $\bar\vfi_s$ ($s\in\bbr$) induced by $-\ol W$ on the total of $\bbr^n$ must be of the form
$$
\bar\vfi_s(y)=\bar y_0(s)+\alpha(s)A(s)(y-y_0(s)),\quad 0\neq\alpha(s)\in\bbr,\ y_0,\bar y_0\in\bbr^n,\  A(s)\in O(n,\bbr).
$$
It then follows that, in this case, the solution $\hat F$ in \eqref{3.2} takes the form
$$
\hat F(p,t)=\bar\vfi(F(p,t),\bar t(t))=\bar y_0(\bar t(t))+\alpha(\bar t(t))A(\bar t(t))(F(p,t)-y_0(\bar t(t))),
$$
giving that {\em $\hat F_t(M)$ is convergent to a round point if and only if $F_t(M)$ is}. Specifically we have
\be\label{7.1}
\hat g=\alpha^2(t)g,\quad\hat h=\alpha(t)Ah,\quad \hat H=\alpha^{-1}(t)AH,
\ee
where $\alpha(t)=\alpha(\bar t(t))$. So {\em for all $t$, $|\hat h|^2\leq c|\hat H|^2$ for a constant $c>0$ if and only if $|h|^2\leq c|H|^2$ for the same $c$}.

Furthermore, the function $\rho$ in the equivalent equation \eqref{3-1} (Theorem \ref{thm3.1}) is exactly $|\alpha(t)|$, which depends only on the parameter $t$. So \eqref{3-1} becomes to \eqref{4} with $a=\alpha^2(t)$.

By making a transformation of time $t\to\td t$ by $\td t:=\int^t_0\alpha^2(\tau)d\tau$, the flow \eqref{4} changes into the standard mean curvature flow \eqref{mcf} with the new time parameter $\td t$. Note that it must hold that $0<\ul a\leq\alpha^2(t)\leq\ol a<+\infty$ (see Lemma \ref{lem5.2}). Then the conclusion of Theorem \ref{thm3.2} comes directly from Theorem \ref{thmab}.
\end{proof}

\section{Some basic evolution formulas}\label{s4}

From this section on, we shall take the ambient space $N$ to be the Euclidean space $\bbr^n$ with the standard flat metric $\ol g$ and the standard coordinates $(y^A)$. For the reader's convenience and the need of the main argument later, we derive in this section the basic evolution formulas for the induced metric $g$, the second fundamental form $h$, the mean curvature $H$, and so on.

For a given $T$: $0<T\leq+\infty$ and a given smooth map $F:M\times[0,T)\to\bbr^n$ satisfying $F_t\in{\mathcal F}(M)$, the pull-back bundle $F^*T\bbr^n\to M\times[0,T)$ decomposes into two orthogonal subbundles: the tangential part ${\mathcal T}=F_{t*}(TM)$ and the normal part ${\mathcal N}=T^\bot_{F_t}M$. The former defines via $F_{t*}$ a ``horizontal distribution'' $\mathcal H$ on $M\times[0,T)$ which can also be defined as (see \cite{a-b} or \cite{b}) $:{\mathcal H}=\{u\in T(M\times[0,T));\ dt(u)=0\}$. Then, according to \cite{a-b}, there are connections $\nabla$ on $\mathcal H$ and $\nabla^\bot$ on $\mathcal N$, respectively, naturally induced by projections from the pull-back connection $\nabla^{F^*T\bbr^n}$. In particular, these two connections are both compatible to the relevant bundle metrics.

Fix a local coordinate system $x^i$ on $M$ and let $\{e_\alpha\}$ be an orthonormal normal frame field of $F(\cdot,t)$. Denote $e_i=\pp{}{x^i}$, $g_{ij}=\lagl F_*(e_i),F_*(e_j)\ragl$, $(g^{ij})=(g_{ij})^{-1}$, and
\begin{align*}
&\ol\nabla_{e_j}(F_*e_i)=\Gamma^k_{ij}(F_*e_k) +h^\alpha_{ij}e_\alpha,\\
\nabla_te_i:=\nabla&_{\pp{}{t}}e_i=\Gamma^j_{it}e_j,\quad \nabla^\bot_te_\alpha:=\nabla^\bot_{\pp{}{t}}e_\alpha =\Gamma^\beta_{\alpha t}e_\beta.
\end{align*}
Then we have
\begin{align}
\ol\nabla_t(F_*e_i)=&\ol\nabla_{e_i}(F_*\pp{}{t})+
F_*([\pp{}{t},e_i])=\ol\nabla_{e_i}(aH) =a_iH+a\ol\nabla_{e_i}H\nnm\\
=&-aA_H(e_i)+a_iH+a\nabla^\bot_{e_i}H,\label{4.1}
\end{align}
implying
\be
F_*(\nabla_t e_i)=(\ol\nabla_t(F_*e_i))^\top =-aA_H(e_i)=-aH^\alpha h^\alpha_{ik}g^{kj}F_*(e_j),\label{eit},
\ee
or, equivalently,
\be\label{gamat}
\Gamma^j_{it}=-aH^\alpha h^\alpha_{ik}g^{kj}.
\ee
Also, by \eqref{4.1},
\begin{align}
\ol\nabla_t e_\alpha=&\lagl\ol\nabla_t e_\alpha,F_*(e_i)\ragl g^{ij} F_*(e_j)+\lagl\ol\nabla_t e_\alpha,e_\beta\ragl e_\beta\nnm\\
=&-\lagl e_\alpha,\left(\ol\nabla_t F_*(e_i)\right)^\bot\ragl g^{ij} F_*(e_j)+\Gamma^\beta_{\alpha t}e_\beta\nnm\\
=&-\lagl e_\alpha,a_iH+a\nabla^\bot_{e_i}H\ragl g^{ij} F_*(e_j)+\Gamma^\beta_{\alpha t}e_\beta\nnm\\
=&-H^\alpha(\nabla a)-a g^{ij}H^\alpha_{,i} F_*(e_j)+\Gamma^\beta_{\alpha t}e_\beta.\label{ealpt}
\end{align}
Moreover, it is easy to see that
\be\label{evogij}
\pp{}{t}g_{ij}=-2a\lagl H,h_{ij}\ragl=-2aH^\beta h^\beta_{ij},
\quad \pp{}{t}g^{ij}=2aH^\beta h^\beta_{kl}g^{ik}g^{lj}.
\ee
To obtain the involution of the second fundamental form $h$, we first find
\begin{align*}\nabla^\bot_t h_{ij}=&\left(\pp{}{t}F_{,ij}\right)^\bot =\left(\pp{}{t}F_{ij}-\Gamma^k_{ij}\pp{}{t}F_k\right)^\bot\\
=&\left(\left(\pp{F}{t}\right)_{ij} -\Gamma^k_{ij}\left(\pp{F}{t}\right)_k\right)^\bot =\left(\left(\pp{F}{t}\right)_{,ij}\right)^\bot\\
=&\left((aH)_{,ij}\right)^\bot =a\left(H_{,ij}-h(A_H(e_i),e_j)\right) +a_{,ij}H+a_iH_{,j}+a_jH_{,i} \\
=&a\left(H_{,ij}-g^{kl}h_{kj}\lagl h_{il},H\ragl\right)+a_{,ij}H+a_iH_{,j}+a_jH_{,i}.
\end{align*}
Then by definition
\begin{align}
\nabla_t h_{ij}=&\nabla^\bot_t(h_{ij})-h_{kj}\Gamma^k_{it} -h_{ik}\Gamma^k_{jt}\nnm\\
=&a(H_{,ij}+H^\beta h^\beta_{jk}h_{il}g^{kl})+a_{,ij}H+a_iH_{,j}+a_jH_{,i}\label{a4.6}
\end{align}
where the formula \eqref{gamat} is used. Note that \eqref{a4.6} can also be obtained by the time-like Codazzi equation given in (18) of \cite{b}. Since
\begin{align}
h_{ij,kl}=&h_{kl,ij} +\big((h^\beta_{kl}h^\beta_{pj} -h^\beta_{kj}h^\beta_{pl})h_{mi} +(h^\beta_{il}h^\beta_{pj} -h^\beta_{ij}h^\beta_{pl})h_{km}\nnm
\\ &-h^\beta_{ki}h^\beta_{lp}h_{jm} +h^\beta_{ki}h^\beta_{jp}h_{lm}\big)g^{pm},
\end{align}
implying
\begin{align}
\Delta h_{ij}=&H_{,ij} +H^\beta h^\beta_{jk}h_{il}g^{kl}\nnm\\
&+\big(2h^\beta_{ki}h^\beta_{jp}h_{lm} -h^\beta_{kj}h^\beta_{pl}h_{mi}
-h^\beta_{ij}h^\beta_{pl}h_{km} -h^\beta_{ki}h^\beta_{lp}h_{jm} \big)g^{kl}g^{pm},
\end{align}
it follows that
\begin{align}
\nabla_t h_{ij}=&a\Delta h_{ij}+a_{,ij}H +a_iH_{,j}+a_jH_{,i} \nnm\\
&+a\big(h^\beta_{kj}h^\beta_{pl}h_{mi}+h^\beta_{ij}h^\beta_{pl}h_{km} +h^\beta_{ki}h^\beta_{lp}h_{jm} -2h^\beta_{ki}h^\beta_{jp}h_{lm}\big)g^{kl}g^{pm},\label{ht1}
\end{align}
implying
\be\nabla^\bot_tH=a(\Delta H+H^\beta h^\beta_{kl}h_{ij}g^{ik}g^{jl})+(\Delta a) H+2a_i H_{,j}g^{ij}.
\ee

From now on, we shall follow the convention of Hamilton (\cite{ha82}) and Huisken (\cite{hui84}) using $S*T$ to denote any linear combination of tensors formed by
contractions, w.r.t. the induced metric $g$, of some given tensors $S$ and $T$. Moreover, to simplify matters, we shall always write $h^2$, $h^3$, $(\nabla h)^2$, $(\nabla h)^3$ and so on for $h*h$, $h*h*h$, $\nabla h*\nabla h$, $\nabla h*\nabla h*\nabla h$ and so on, accordingly. Thus by \eqref{ht1} it holds that
\be\label{ht2}
\nabla_th_{ij}=a\Delta h_{ij}+(\nabla^2a*h+\nabla a*\nabla h +a*h^3)_{ij}.
\ee
Furthermore, because
\begin{align*}
\Delta|h|^2=&2\sum g^{ik}g^{jl}h^\alpha_{kl}\Delta h^\alpha_{ij}+2|\nabla h|^2\\
=&2\sum g^{ik}g^{jl}h^\alpha_{kl} H^\alpha_{,ij} +2\sum g^{ik}g^{jl}h^\alpha_{kl}\big(h^\alpha_{mi}(h^\beta_{rs}h^\beta_{pj} -h^\beta_{rj}h^\beta_{ps}) +h^\alpha_{rm}(h^\beta_{is}h^\beta_{pj} -h^\beta_{ij}h^\beta_{ps})\nnm\\ &-h^\beta_{ri}(h^\beta_{sp}h^\alpha_{jm} -h^\beta_{jp}h^\alpha_{sm})\big)g^{rs}g^{pm}+2|\nabla h|^2\\
=&2\sum g^{ik}g^{jl}h^\alpha_{kl} H^\alpha_{,ij} +2H^\beta h^\beta_{pj}h^\alpha_{kl}h^\alpha_{mi}g^{ik}g^{jl}g^{pm} +2|\nabla h|^2-2R_1
\end{align*}
where
\begin{align*}
R_1:=&\sum_{\alpha,\beta}\left(\sum_{i,j,k,l} g^{ik}g^{jl}h^\alpha_{ij}h^\beta_{kl}\right)^2 \\ &+\sum g^{ij}g^{kp}g^{lq}g^{rs}(h^\beta_{ik}h^\alpha_{pr} -h^\alpha_{ik}h^\beta_{rp})(h^\beta_{jl}h^\alpha_{qs} -h^\alpha_{jl}h^\beta_{qs})\\
=&\sum_{\alpha,\beta}\lagl h^\alpha,h^\beta\ragl^2 +|\strl{\bot}{R}|^2= h^4
\end{align*}
with $|\strl{\bot}{R}|$ being the norm of the normal curvature operator,
it follows that
\begin{align}
\pp{}{t}|h|^2=&a\Delta|h|^2-2a|\nabla h|^2 +2aR_1+2\sum_{\alpha,i,j,k,l}g^{ik}g^{jl}h^\alpha_{ij}(H^\alpha a_{,kl} +a_kH^\alpha_{,l}+a_lH^\alpha_{,k})\nnm\\
=&a\Delta|h|^2-2a|\nabla h|^2+2a R_1 +2\sum_{\alpha,i,j,k,l}g^{ik}g^{jl}H^\alpha h^\alpha_{ij}a_{,kl}
+4\sum_{\alpha,i,j,k,l}g^{ik}g^{jl}h^\alpha_{ij}a_kH^\alpha_{,l}\nnm\\
=&a\Delta|h|^2-2a|\nabla h|^2+\nabla^2a*h^2+\nabla a*h*\nabla h+a*h^4,\\
\pp{}{t}|H|^2=&a\Delta|H|^2 -2a|\nabla H|^2+2a R_2 +2|H|^2\Delta a+2\lagl\nabla a,\nabla |H|^2\ragl\nnm\\
=&a\Delta|H|^2-2a|\nabla H|^2+\nabla^2a*h^2+\nabla a*h*\nabla h+a*h^4.
\end{align}
with the notation
\be\label{4.13}
R_2:=\sum_{i,j,k,l}\sum_{\alpha,\beta}g^{ik}g^{jl} H^\alpha h^\alpha_{ij}H^\beta h^\beta_{kl}=|\lagl H,h\ragl|^2\leq|H|^2|h|^2.
\ee

\section{Higher derivative estimates and the blow-up theorem}

This section is the main part of the present paper and is devoted to prove the following main theorem:

\begin{thm}\label{thm5.1}
Under the assumption of Theorem \ref{main1}, the conformal mean curvature flow \eqref{4} has a unique solution on a finite maximal time interval $0\leq t<T<+\infty$. Moreover, $\max_{M}|h|^2\to\infty$ as $t\to T$.
\end{thm}

First of all, note that the existence and the uniqueness of the maximal solution are given by Theorem \ref{exiuni}. Thus we need to prove the finiteness of the maximal time interval $[0,T)$ and the blow-up of $|h|^2$.

For any $\bbr^n$-valued map $F$, we denote by $|F|^2$ the square norm of the position vector $F$.

\begin{lem}\label{lem5.2}
The function $|F|^2$ is bounded on $M\times[0,T)$. In particular, the image $F(M\times[0,T))$ of $F$ is included in a bounded domain of $\bbr^{m+p}$ where restrictions of the function $a$ and all of its derivatives on the ambient $\bbr^n$ are bounded, that is, $|\ol\nabla^i a|^2\leq\ol A_i$ for some constant $\ol A_i$, $i=0,1,2,\cdots$.
\end{lem}

\begin{proof}
It is easily found that $\pp{}{t}|F|^2=a\Delta|F|^2-2ma<a\Delta|F|^2$. It then follows from the maximum value principle that $|F|^2$ is bounded from above by the maximal value of it on the initial submanifold $F_0$.
\end{proof}
\begin{lem}
The maximal time of existence $T$ is finite.
\end{lem}

\begin{proof}
Once again we use $\pp{}{t}|F|^2=a\Delta|F|^2-2ma$. By the previous lemma, it holds that $\ul a\leq a\leq \ol a$ for some $\ul a,\ol a>0$. So we have $\pp{}{t}|F|^2\leq a\Delta|F|^2-2m\ul a$ which gives that $$\pp{}{t}(|F|^2+2m\ul at)-a\Delta(|F|^2+2m\ul at)\leq 0.$$
Then the maximum value principle shows that, for any $t\in[0,T)$,
$$2m\ul at\leq|F|^2+2m\ul at\leq \max_M|F_0|^2.$$
Letting $t\to T$ we have that
$$T\leq \fr1{2m\ul a}\max_M|F_0|^2.$$
\end{proof}

Next we are to prove the blow-up part of Theorem \ref{thm5.1}. Before doing this, we have to give some estimates for the higher order derivatives of the second fundamental form and then those of the solution $F$ itself. But these estimates rely on higher order derivatives $\nabla^i a$ of the composed function $a\equiv a\circ F$.

To proceed, we need the following identities which are derived in \cite{a-b} (see also \cite{b}):
 \begin{align}
R(e_i,e_j,e_k,e_l)=&\lagl h_{il},h_{jk}\ragl-\lagl h_{ik},h_{jl}\ragl= h^2,\label{rijkl}\\
R^\bot(e_\alpha,e_\beta,e_i,e_j)=&g(A_\alpha(e_j),A_\beta(e_i))-g(A_\alpha(e_i),A_\beta(e_j))=h^2,\label{rabij}\\
R(\partial_t,e_i,e_j,e_k)=&\lagl \nabla^\bot_{e_k} F_t,h_{ij}\ragl-\lagl \nabla^\bot_{e_j} F_t,h_{ik}\ragl\label{rtijk1}\\
=& \nabla a*h^2+a*h*\nabla h,\label{rtijk2}\\
R^\bot(\partial_t,e_i,e_\alpha,e_\beta)=&\lagl \nabla^\bot_{A_\alpha(e_i)} F_t,e_\beta\ragl -\lagl \nabla^\bot_{A_\beta(e_i)} F_t,e_\alpha\ragl\label{rtiab1}\\
=&\nabla a*h^2+a*h*\nabla h.\label{rtiab2}
\end{align}

The following Young's inequality is frequently used in our estimation later:

\begin{lem}[Young's inequality]\label{young}
Let $a$ and $b$ be two nonnegative real numbers and $p$ and $q$ be positive real numbers such that $1/p +1/q = 1$. Then
$$ab\leq \veps^p{\frac { a^{p}}{p}}+\frac1{\veps^q}{\frac {b^{q}}{q}},\quad\forall\veps>0.$$
The equality holds if and only if $\veps^{p+q}a^p =b^q$. In particular, we have the following so-called Peter-Paul inequality:
$$2ab\leq\veps a^2+\fr1\veps b^2$$
for any $\veps>0$.
\end{lem}

\begin{lem}\label{lem nabla a} Let $a\equiv a\circ F$. Then
for any $l\geq 0$ it holds that
\begin{align}
\nabla^{l+2}a=&\sum_{p=1}^{l+2}\sum_{r_1+\cdots+r_p =l-p+2}\lagl\ol\nabla^p a,\nabla^{r_1+1}F\otimes\cdots\otimes\nabla^{r_p+1}F\ragl\label{nabla a1}\\
=&\sum_Aa_A\nabla^{l+2}F^A+\sum_{p=2}^{l+2}\sum_{r_1+\cdots+r_p =l-p+2}\lagl\ol\nabla^p a,\nabla^{r_1+1}F\otimes\cdots\otimes\nabla^{r_p+1}F\ragl\label{nabla a2}
\end{align}
and, for $k\geq 0$,
\begin{align}
\nabla^{l+2}F=&\nabla^l h+\sum_{\iota=0}^{k-1}(*^{2(k-\iota)}_{2\iota+1}h)^iF_*(e_i) +\sum_{\iota=0}^{k-1}(*^{2(k-\iota)+1}_{2\iota}h)^\alpha e_\alpha,\quad \text{if }l=2k;\label{nabla f1}\\
\nabla^{l+2}F=&\nabla^l h+\sum_{\iota=0}^k(*^{2(k-\iota+1)}_{2\iota}h)^iF_*(e_i) +\sum_{\iota=0}^{k-1}(*^{2(k-\iota)+1}_{2\iota+1}h)^\alpha e_\alpha,\quad \text{if }l=2k+1\label{nabla f2}
\end{align}
where, for integers $p\geq 1$ and $q\geq 0$,
$$
*^0_qh=1,\quad*^p_qh=\sum_{r_1+\cdots+r_p=q}\nabla^{r_1}h*\cdots*\nabla^{r_p}h.
$$
\end{lem}

\begin{proof} First we prove \eqref{nabla a1}.
For $l=0$, we have
\begin{align*}
(\nabla^2a)_{ij}=&\sum_{A,B} (\ol\nabla^2a)_{AB}F^A_iF^B_j+\sum_A(\ol\nabla a)_AF^A_{,ij}\\
=&\sum_{p=1}^{2}\sum_{r_1+\cdots+r_p =-p+2}\lagl\ol\nabla^p a,
\nabla^{r_1+1}F\otimes\cdots\otimes\nabla^{r_p+1}F\ragl_{ij}.
\end{align*}
Suppose the formula is true for $l-1\geq 0$, that is
$$
\nabla^{l+1}a=\sum_{p=1}^{l+1}\sum_{r_1+\cdots+r_p =l-p+1}\lagl\ol\nabla^p a,\nabla^{r_1+1}F\otimes\cdots\otimes\nabla^{r_p+1}F\ragl.
$$
then for $l\geq 1$
\begin{align*}
\nabla^{l+2}a=&\sum_{p=1}^{l+1}\sum_{r_1+\cdots+r_p =l-p+1}\nabla\lagl\ol\nabla^p a,\nabla^{r_1+1}F\otimes\cdots\otimes\nabla^{r_p+1}F\ragl\\
=&\sum_{p=1}^{l+1}\sum_{r_1+\cdots+r_p =l-p+1}\lagl\nabla\ol\nabla^p a,\nabla^{r_1+1} F\otimes\cdots\otimes\nabla^{r_p+1}F\ragl \\
&+\sum_{p=1}^{l+1}\sum_{r_1+\cdots+r_p =l-p+1}\lagl\ol\nabla^p a,\nabla(\nabla^{r_1+1}F\otimes\cdots\otimes\nabla^{r_p+1}F)\ragl\\
=&\sum_{p=1}^{l+2}\sum_{r_1+\cdots+r_p =l-p+2}\lagl\ol\nabla^p a,\nabla^{r_1+1}F\otimes\cdots\otimes\nabla^{r_p+1}F\ragl\label{nabla a1}.
\end{align*}
Thus formula \eqref{nabla a1} holds for any $l\geq 0$.

Now we use induction again to prove formulas \eqref{nabla f1} and \eqref{nabla f2}. For $l=0$, we have
$$
(\nabla^2 F)_{ij}=(\nabla_{e_j} F_*)(e_i)=h_{ij}.
$$
Suppose \eqref{nabla f1} holds for $k-1\geq 0$, that is, for $l=2(k-1)\geq 0$,
$$
\nabla^{2(k-1)+2}F=\nabla^{2(k-1)} h+\sum_{\iota=0}^{k-1-1}(*^{2(k-1-\iota)}_{2\iota+1}h)^iF_*(e_i) +\sum_{\iota=0}^{k-1-1}(*^{2(k-1-\iota)+1}_{2\iota}h)^\alpha e_\alpha.
$$
Then for $l=2(k-1)+1$ we find
\begin{align*}
\nabla^{2(k-1)+1+2}F=&\ol\nabla\nabla^{2(k-1)} h+\sum_{\iota=0}^{k-2}\nabla(*^{2(k-\iota-1)}_{2\iota+1}h)^iF_*(e_i)
+\sum_{\iota=0}^{k-2}(*^{2(k-\iota-1)}_{2\iota+1}h)^i(\ol\nabla F_*(e_i) )^\bot\\
&+\sum_{\iota=0}^{k-2}\nabla(*^{2(k-\iota)-1}_{2\iota}h)^\alpha e_\alpha +\sum_{\iota=0}^{k-2}(*^{2(k-\iota)-1}_{2\iota}h)^\alpha(\ol\nabla e_\alpha)^\top\\
=&\nabla^{2(k-1)+1} h+(\nabla^{2(k-1)} h*h)^iF_*(e_i)+\sum_{\iota=0}^{k-2}(*^{2(k-\iota-1)}_{2\iota+2}h)^iF_*(e_i) \\
&+\sum_{\iota=0}^{k-2}(*^{2(k-\iota-1)}_{2\iota+1}h*h)^\alpha e_\alpha+\sum_{\iota=0}^{k-2}(*^{2(k-\iota)-1}_{2\iota+1}h)^\alpha e_\alpha  \\
&+\sum_{\iota=0}^{k-2}(*^{2(k-\iota)-1}_{2\iota}h*h)^i F_*(e_i)\\
=&\nabla^{2(k-1)+1} h+\sum_{\iota=0}^{k-1}(*^{2(k-1-\iota+1)}_{2\iota}h)^iF_*(e_i) +\sum_{\iota=0}^{k-1-1}(*^{2(k-1-\iota)+1}_{2\iota+1}h)^\alpha e_\alpha.
\end{align*}
So \eqref{nabla f2} holds for $k-1\geq 0$, from which it follows that
\begin{align*}
\nabla^{2k+2}F=&\ol\nabla\nabla^{2(k-1)+1} h+\sum_{\iota=0}^{k-1}\nabla(*^{2(k-\iota)}_{2\iota}h)^iF_*(e_i)
+\sum_{\iota=0}^{k-1}(*^{2(k-\iota)}_{2\iota}h)^i(\ol\nabla F_*(e_i) )^\bot\\
&+\sum_{\iota=0}^{k-2}\nabla(*^{2(k-\iota)-1}_{2\iota+1}h)^\alpha e_\alpha +\sum_{\iota=0}^{k-2}(*^{2(k-\iota)-1}_{2\iota+1}h)^\alpha(\ol\nabla e_\alpha)^\top\\
=&\nabla^{2(k-1)+2} h+(\nabla^{2(k-1)+1} h*h)^iF_*(e_i)+\sum_{\iota=0}^{k-1}(*^{2(k-\iota)}_{2\iota+1}h)^iF_*(e_i) \\
&+\sum_{\iota=0}^{k-1}(*^{2(k-\iota)}_{2\iota}h*h)^\alpha e_\alpha+\sum_{\iota=0}^{k-2}(*^{2(k-\iota)-1}_{2\iota+2}h)^\alpha e_\alpha +\sum_{\iota=0}^{k-2}(*^{2(k-\iota)-1}_{2\iota+1}h*h)^i F_*(e_i)\\
=&\nabla^{2k} h+\sum_{\iota=0}^{k-1}(*^{2(k-\iota)}_{2\iota+1}h)^iF_*(e_i) +\sum_{\iota=0}^{k-1}(*^{2(k-\iota)+1}_{2\iota}h)^\alpha e_\alpha.
\end{align*}
Therefore \eqref{nabla f1} holds for all $k\geq 1$. By the principle of induction, both \eqref{nabla f1} and \eqref{nabla f2} are proved.
\end{proof}

\begin{prop}
The evolution of the $l$-th covariant derivative of $h$ is of the form
\be
\nabla_t\nabla^lh=a\Delta\nabla^lh+\sum_{r_0+r_1=l+2,r_0\geq 1}\nabla^{r_0}a*\nabla^{r_1}h +\sum_{r_0+r_1+r_2+r_3=l}\nabla^{r_0}a*\nabla^{r_1}h*\nabla^{r_2}h*\nabla^{r_3}h.
\ee
\end{prop}
\begin{proof}
We prove this proposition by induction. When $l=0$, it is easy to see from \eqref{ht2} that
\begin{align*}
\nabla_th^\alpha_{ij}=& a\Delta h^\alpha_{ij}+(\nabla^2a*h+\nabla a*\nabla h+ah^3)^\alpha_{ij}.
\end{align*}

Now suppose the conclusion holds for $l-1\geq 0$. Then by the time-like Ricci identity, \eqref{rtijk2} and \eqref{rtiab2}, we find
\begin{align*}
(\nabla_t\nabla^lh)^\alpha_{i_1\cdots i_{l+2}}=&(\nabla\nabla^lh)^\alpha_{i_1\cdots i_{l+2}t}
=(\nabla\nabla^lh)^\alpha_{i_1\cdots i_{l+1}ti_{l+2}} +\sum_i(\nabla^{l-1}h)^\alpha_{ii_2\cdots i_{l+1}}R_{i_1ii_{l+2}t}\\ &+\cdots+\sum_i(\nabla^{l-1}h)^\alpha_{i_1i_2\cdots i_li}R_{i_{l+1}ii_{l+2}t} -\sum_\beta(\nabla^{l-1}h)^\beta_{i_1i_2\cdots i_{l+1}}R_{\beta\alpha i_{l+2}t}\\
=&\nabla_{i_{l+2}}(\nabla_t\nabla^{l-1}h)^\alpha_{i_1\cdots i_{l+1}}+\sum_{r_0+r_1+r_2+r_3=l} (\nabla^{r_0}a*\nabla^{r_1}h*\nabla^{r_2}h*\nabla^{r_3}h)^\alpha_{i_1\cdots i_{l+2}}\\
=&\nabla_{i_{l+2}}\Big(a\Delta\nabla^{l-1}h+\sum_{r_0+r_1=l+1,r_0\geq 1}\nabla^{r_0}a*\nabla^{r_1}h\\
&+\sum_{r_0+r_1+r_2+r_3=l-1} \nabla^{r_0}a*\nabla^{r_1}h*\nabla^{r_2}h*\nabla^{r_3}h\Big)^\alpha_{i_1\cdots i_{l+1}}\\
&+\sum_{r_0+r_1+r_2+r_3=l} (\nabla^{r_0}a*\nabla^{r_1}h*\nabla^{r_2}h*\nabla^{r_3}h)^\alpha_{i_1\cdots i_{l+2}}\\
=&\nabla_{i_{l+2}}(a\Delta\nabla^{l-1}h)^\alpha_{i_1\cdots i_{l+1}}+\sum_{r_0+r_1=l+2,r_0\geq 1}(\nabla^{r_0}a*\nabla^{r_1}h)^\alpha_{i_1\cdots i_{l+2}}\\
&+\sum_{r_0+r_1+r_2+r_3=l} (\nabla^{r_0}a*\nabla^{r_1}h*\nabla^{r_2}h*\nabla^{r_3}h)^\alpha_{i_1\cdots i_{l+2}}.
\end{align*}

On the other hand, for any $S\in\Gamma(\otimes^r{\mathcal H^*}\otimes{\mathcal N})$, we have the following formula of commuting the Laplacian and gradient:
\begin{align*}
\nabla_k(\Delta S)^\alpha_{i_1\cdots i_r}=&S^\alpha_{i_1\cdots i_rjjk} =S^\alpha_{i_1\cdots i_rjkj}+S^\alpha_{ii_2\cdots i_rj}R^i_{i_1jk} +\cdots +S^\alpha_{i_1\cdots i_ri}R^i_{jjk}-S^\beta_{i_1\cdots i_rj}R^\alpha_{\beta jk}\\
=&\left(S^\alpha_{i_1\cdots i_rkj}+S^\alpha_{ii_2\cdots i_r}R^i_{i_1jk} +\cdots +S^\alpha_{i_1\cdots i_{r-1}i}R^i_{i_rjk}-S^\beta_{i_1\cdots i_r}R^\alpha_{\beta jk}\right)_{,j}\\
&+(\nabla S*h^2)^\alpha_{i_1\cdots i_rk}\\
=&S^\alpha_{i_1\cdots i_rkjj}+\nabla_j(S*h^2)^\alpha_{i_1\cdots i_rkj}+(\nabla S*h^2)^\alpha_{i_1\cdots i_rk}\\
=& \Delta(\nabla_k S)^\alpha_{i_1\cdots i_r}+(\nabla S*h^2+S*h*\nabla h )^\alpha_{i_1\cdots i_rk}.
\end{align*}
In particular, inserting $S=\nabla^{l-1}h$ we have
\begin{align*}
\nabla_{i_{l+2}}(a\Delta\nabla^{l-1}h)^\alpha_{i_1\cdots i_{l+1}}
=&a_{i_{l+2}} (\Delta\nabla^{l-1}h)^\alpha_{i_1\cdots i_{l+1}}+ a\Delta(\nabla^lh)^\alpha_{i_1\cdots i_{l+2}}\\
&+a(\nabla^lh*h^2+\nabla^{l-1}h*h*\nabla h)^\alpha_{i_1\cdots i_{l+2}}.
\end{align*}
Thus we finally obtain
\begin{align}\label{evohiderh}
\nabla_t\nabla^lh=&a\Delta\nabla^lh+\sum_{r_0+r_1=l+2,r_0\geq 1}\nabla^{r_0}a*\nabla^{r_1}h\nnm\\
&+\sum_{r_0+r_1+r_2+r_3=l} \nabla^{r_0}a*\nabla^{r_1}h*\nabla^{r_2}h*\nabla^{r_3}h.
\end{align}
\end{proof}

\begin{cor}\label{evonhiderh}
The evolution of $|\nabla^lh|^2$ is of the form
\begin{align*}
\pp{}{t}|\nabla^lh|^2=&a\Delta|\nabla^lh|^2-2a|\nabla^{l+1}h|^2+\sum_{r_0+r_1=l+2,r_0\geq 1}\nabla^{r_0}a*\nabla^{r_1}h*\nabla^lh\\ &+\sum_{r_0+r_1+r_2+r_3=l}\nabla^{r_0}a*\nabla^{r_1}h*\nabla^{r_2}h*\nabla^{r_3}h*\nabla^lh.
\end{align*}
\end{cor}

Now we are to make estimations first for all the higher order derivatives of the second fundamental form $h$, and then for those of the $\bbr^n$-valued function $F$.

\begin{prop}\label{prop5.7}
Suppose that the conformal mean curvature flow \eqref{4} has a solution on a time interval $t\in[0,\tau]$. If $|h|^2$ is bounded on $[0,\tau]$, say, $|h|^2\leq C^0_0$, then for each $l\geq 1$, it holds that $|\nabla^lh|^2\leq C^0_l(1+1/t^l)$ for all $t\in (0,\tau]$, where $C^0_l$ is a constant that only depends on $m,l,C^0_0$ and the bounds of $\ol\nabla^ka$, $k=0,1,\cdots,l+2$.
\end{prop}

\begin{proof} The idea of proving Lemma \ref{prop5.7} comes from \cite{a-b}.
For each $l\geq 1$ define
$$G_l=t^l|\nabla^l h|^2+\fr l{\ul a}t^{l-1}|\nabla^{l-1}h|^2,$$
where $\ul a$ is a positive lower bound of $a$.

We shall prove the lemma by induction on $l$. For the case $l=1$, we first use the assumption $|h|^2\leq C^0_0$ and Lemma \ref{lem nabla a} to deduce that
\be
|\nabla^2 a|^2\leq C_1,\quad
|\nabla^3 a|^2\leq C_2|\nabla h|^2+C_3
\ee
where $C_1, C_2$ and $C_3$ are dependent only on $C^0_0$, the dimension $m$ and the bounds of $\ol\nabla a, \ol\nabla^2a, \ol\nabla^3 a$. Then by Corollary \ref{evonhiderh}
\begin{align*}
\pp{}{t}G_1=&a\Delta G_1+|\nabla h|^2+t\left(-2a|\nabla^2 h|^2 +\nabla^3a*h*\nabla h+\nabla^2a*(\nabla h)^2\right.\\
&\left.+\nabla a*\nabla h*\nabla^2 h +\nabla a*h^3*\nabla h+a*h^2*(\nabla h)^2\right)\\
&+\fr1{\ul a}\left(-2a|\nabla h|^2+\nabla^2a*h^2 +\nabla a*h*\nabla h+a*h^4\right).
\end{align*}
So at points where $|\nabla h|^2\leq 1$, since
$$|\nabla^3 a|^2\leq C_2|\nabla h|^2+C_3\leq C,\text{ and }t<T<+\infty,$$
it holds by Young's inequality that
\be
\pp{}{t}G_1\leq a\Delta G_1+t(-2a+\veps)|\nabla^2 h|^2+c_2\leq a\Delta G_1+c_2
\ee
for $\veps$ small enough, where $c_2$ is a constant that only depends on $m,C^0_0$ and the bounds of $a$, $\ol\nabla a$, $\ol\nabla^2 a$ and $\ol\nabla^3 a$;
while at points where $|\nabla h|^2\geq 1$, since
$$|\nabla^3 a|^2\leq C_2|\nabla h|^2+C_3\leq (C|\nabla h|)^2,$$
it also holds by Young's inequality that
\be\label{pt g1}
\pp{}{t}G_1\leq a\Delta G_1+(-\fr a{\ul a}+\veps+tc_1)|\nabla h|^2+t(-2a+\veps)|\nabla^2 h|^2+c_2 \leq a\Delta G_1+c_2
\ee
for some positive $\veps\leq\min\{\fr12, 2\ul a\}$ and $t\leq\fr{1-\veps}{c_1}$, where $c_1$ and $c_2$ are constants that only depend on $m,C^0_0$ and the bounds of $a$, $\ol\nabla a$, $\ol\nabla^2 a$ and $\ol\nabla^3 a$. So we always have \eqref{pt g1}, if $t\leq\fr{1-\veps}{c_1}$. Therefore it follows from the maximal value principle that
$$t|\nabla h|^2\leq G_1\leq \ul a^{-1}C^0_0+c_2t$$
or $$|\nabla h|^2\leq\fr1t G_1\leq \fr{C^0_0}{\ul at}+c_2\leq C^0_1(1+\fr1t)$$
on the interval $(0,\fr{1-\veps}{c_1}]$. When $t>\fr{1-\veps}{c_1}$, we can consider the internal $(t-\fr{1-\veps}{2c_1},t+\fr{1-\veps}{2c_1}]$ of length $\fr{1-\veps}{c_1}$ on which similar argument can give a similar estimation for $G_1$. Since $\veps$ can be chosen fixed, we can cover $(\fr{1-\veps}{c_1}-\delta,\tau]$, $0<\delta\leq\fr{1-\veps}{2c_1}$, with a family of such intervals of a fixed length. Due to the finiteness of $T$ and the fact that $\tau<T$, this consideration will directly lead to the conclusion for $l=1$.

Now we suppose the conclusion is true for less than or equal to $l-1\geq 1$. Then $|h|^2,\cdots |\nabla^{l-1}h|^2$ are all bounded from above. By using this and Lemma \ref{lem nabla a}, we get
\be
\max\{a,|\nabla a|^2,\cdots,|\nabla^{l+1} a|^2\}\leq C_1,\quad
|\nabla^{l+2} a|^2\leq C_2|\nabla^l h|^2+C_3
\ee
where $C_1, C_2$ and $C_3$ are dependent only on $C^0_0$, the dimension $m$ and the bounds of $a,\ol\nabla a,\cdots, \ol\nabla^{l+2} a$. Once more we use Corollary \ref{evonhiderh} to find
\begin{align*}
\pp{}{t}G_l=&a\Delta G_l+lt^{l-1}|\nabla^l h|^2+t^l\big(-2a|\nabla^{l+1} h|^2\\
&+\sum_{r_0+r_1=l+2,r_0\geq 1}\nabla^{r_0}a*\nabla^{r_1}h*\nabla^lh\\ &+\sum_{r_0+r_1+r_2+r_3=l}\nabla^{r_0}a*\nabla^{r_1}h* \nabla^{r_2}h*\nabla^{r_3}h*\nabla^lh\big)\\
&+\fr{l(l-1)}{\ul a}t^{l-2}|\nabla^{l-1}h|^2 +\fr l{\ul a}t^{l-1}\big(-2a|\nabla^l h|^2\\
&+\sum_{r_0+r_1=l+1,r_0\geq 1}\nabla^{r_0}a*\nabla^{r_1}h*\nabla^{l-1}h \\ &+\sum_{r_0+r_1+r_2+r_3=l-1}\nabla^{r_0}a*\nabla^{r_1}h* \nabla^{r_2}h*\nabla^{r_3}h*\nabla^{l-1}h\big).
\end{align*}
As in the case $l=1$ we can use Young's inequality to obtain that
\be
\pp{}{t}G_l\leq a\Delta G_l+t^l(-2a+\veps )|\nabla^{l+1} h|^2+c_4
\leq a\Delta G_l+c_4
\ee
at points where $|\nabla^lh|^2\leq 1$, and
\be
\pp{}{t}G_l\leq a\Delta G_l+lt^{l-1}\big(-\fr a{\ul a}+\veps +tc_3\big)|\nabla^l h|^2+t^l(-2a+\veps )|\nabla^{l+1} h|^2+c_4
\leq a\Delta G_l+c_4
\ee
at points where $|\nabla^lh|^2\geq 1$, for some positive $\veps\leq \min{\fr12, 2\ul a}$ and $t\leq\fr{1-\veps}{c_3}$, where $c_3$ and $c_4$ are constants that only depend on $m,C^0_0$ and the bounds of $a, \ol\nabla a,\cdots,\ol\nabla^{l+2} a$. Thus the maximal value principle gives that
$$t^l|\nabla^l h|^2\leq G_l\leq c_4t, \text{\ or\ } |\nabla^l h|^2\leq \fr{c_4}{t^{l-1}}\leq C^0_l\Big(1+\fr1{t^{l-1}}\Big)$$
on the interval $(0,\fr{1-\veps}{c_3}]$. When $t>\fr{1-\veps}{c_3}$, we fix a small $\veps>0$ and consider a family of intervals of fixed length no more than $\fr{1-\veps}{c_3}$ similarly as in the case of $l=1$ to reach the conclusion for $l$.
\end{proof}

From Lemma \ref{lem nabla a} and Lemma \ref{prop5.7} we have directly the following corollary:

\begin{cor}\label{cor5.8} If $\max |h|^2<+\infty$, then there exist constants $C'(l)$, $C''(l)$ and $C'''(l)$ such that
\be\label{5.19}
|\nabla^lh|^2\leq C'(l),\quad |\nabla^lF|^2\leq C''(l)\quad \text{and thus} \quad |\nabla^l a|^2\leq C'''(l),\quad l\geq 0.
\ee
\end{cor}

In order to prove Theorem \ref{thm5.1}, we follow \cite{b} to fix a smooth metric $\td g$ on $M$ with the Levi-Civita connection $\td\nabla$, which can trivially extend to a time-independent metric on $M\times [0,T)$, still denoted by $\td g$. Then we need to use corresponding induced connections, also denoted by $\td\nabla$, on the relevant bundles on $M\times [0,T)$ for some computations. For example, it is easy to find that $\td\nabla_t g=-2a\lagl H,h\ragl$ by which we have
$$
\left|\pp{}{t}\left(\fr{g(v,v)}{\td g(v,v)}\right)\right|=\left|\fr{(\td\nabla_t g)(v,v)}{g(v,v)}\fr{g(v,v)}{\td g(v,v)}\right| \leq 2a|H||h|_g\fr{g(v,v)}{\td g(v,v)},\quad\forall\,v\in TM,
$$
giving that (\cite{b})
\be\label{32}
\fr1c\td g\leq g\leq c\td g
\ee
for some constant $c>0$. Thus any estimation of a length function with respect to the metric $\td g$ is equivalent to that with respect to the metric $g$.

Now let $\td  T=\td\nabla-\nabla$ be the difference of the two connections. Then $\td  T\in\Gamma(\mathcal{H}^*\otimes \mathcal{H}^*\otimes \mathcal{H})$. Moreover, for any section $S$ of a bundle, constructed from the induced bundle $F^*T\bbr^{m+p}$, the horizontal distribution ${\mathcal H}\subset T(M\times [0,T))$ and the normal subbundle ${\mathcal N}\subset F^*T\bbr^{m+p}$, we have that $\td\nabla S-\nabla S=S*\td  T$.

\begin{lem}\label{lem5.9} It holds that
\be
\td\nabla^{l+1} F=\nabla^{l+1} F+\sum_{q=1}^l\sum_{p=1}^{l+1-q}(*^p_{l-p-q+1}\td  T)*\nabla^q F,\quad l\geq 0.\label{tdnabla f1}
\ee
where, for integers $p\geq 1$ and $q\geq 0$,
$$
*^p_q\td  T=\sum_{r_1+\cdots+r_p=q}\td\nabla^{r_1}\td  T*\cdots*\td\nabla^{r_p}\td  T.
$$
\end{lem}

\begin{proof} We prove \eqref{tdnabla f1} by induction on $l$. If $l=0$, then both sides of \eqref{tdnabla f1} are equal to $dF$. So \eqref{tdnabla f1} holds for $l=0$.

Suppose that \eqref{tdnabla f1} holds for $l=k$. Then when $l=k+1$, we compute
\begin{align*}
\td\nabla^{k+1+1} F=&\td\nabla(\td\nabla^{k+1}F)=(\nabla+\td T)\left(\nabla^{k+1}F +\sum_{q=1}^{k}\sum_{p=1}^{k+1-q}(*^p_{k-p-q+1}\td  T)*\nabla^qF\right)\\
=&(\nabla+\td  T)\nabla^{k+1}F +(\nabla+\td  T)\sum_{q=1}^k\sum_{p=1}^{k+1-q} (*^p_{k-p-q+1}\td  T)*\nabla^qF \\
=&\nabla^{k+2}F+\td  T\nabla^{k+1}F +\nabla\sum_{q=1}^k\sum_{p=1}^{k+1-q} (*^p_{k-p-q+1}\td  T)*\nabla^qF \\
&+\td T\sum_{q=1}^k\sum_{p=1}^{k+1-q} (*^p_{k-p-q+1}\td  T)*\nabla^qF \\
=&\nabla^{k+2}F+\td  T\nabla^{k+1}F +\sum_{q=1}^k\sum_{p=1}^{k+1-q} (\td\nabla-\td T)(*^p_{k-p-q+1}\td  T)*\nabla^qF \\
& +\sum_{q=1}^k\sum_{p=1}^{k+1-q}(*^p_{k-p-q+1}\td  T)*\nabla^{q+1}F +\td T\sum_{q=1}^k\sum_{p=1}^{k+1-q} (*^p_{k-p-q+1}\td  T)*\nabla^qF \\
=&\nabla^{k+1+1}F+\td  T\nabla^{k+1}F +\sum_{q=1}^k\sum_{p=1}^{k+1-q} \Big((*^p_{k+1-p-q+1}\td  T)*\nabla^qF\\
& +(*^p_{k-p-q+1}\td  T)*\nabla^{q+1}F +(*^p_{k-p-q+1}\td  T*\td  T)*\nabla^qF\Big)\\
=&\nabla^{k+1+1}F+\sum_{q=1}^{k+1}\sum_{p=1}^{k+2-q} (*^p_{k+1-p-q+1}\td  T)*\nabla^qF -\sum_{q=1}^k(*^{k+2-q}_0\td  T)*\nabla^qF\\
&+\sum_{q=1}^k\sum_{p=1}^{k+1-q} (*^p_{k-p-q+1}\td  T)*\nabla^{q+1}F \\
&+\sum_{q=1}^k(*^{k+1-q}_0\td  T*\td T)*\nabla^qF +\sum_{q=1}^k\sum_{p=1}^{k-q}(*^p_{k-p-q+1}\td  T*\td T)*\nabla^qF\\
=&\nabla^{k+1+1}F +\sum_{q=1}^{k+1}\sum_{p=1}^{k+2-q} (*^p_{k+1-p-q+1}\td  T)*\nabla^qF +\sum_{q=2}^{k+1}\sum_{p=1}^{k+2-q}(*^p_{k+1-p-q+1}\td  T)*\nabla^qF\\
&+\sum_{q=1}^k\sum_{p=1}^{k-q}(*^p_{k-p-q+1}\td  T*\td T)*\nabla^qF \\
=&\nabla^{k+1+1}F +\sum_{q=1}^{k+1}\sum_{p=1}^{k+1+1-q} (*^p_{k+1-p-q+1}\td  T)*\nabla^qF.
\end{align*}
Therefore, \eqref{tdnabla f1} holds for $l=k+1$ and thus holds for all $l\geq 0$.
\end{proof}

\begin{rmk} In \cite{b}, the author gives a different expression for $\td\nabla^l F$ as follows:
\begin{align}
\td\nabla^l F=&F_*\td\nabla^{l-2}\td  T+F_*\left(\sum_{i_0+2i_1+\cdots+(l-2)i_{l-3}=l-1} \td  T^{i_0}*(\td\nabla \td  T)^{i_1}* \cdots *(\td\nabla^{l-3} \td  T)^{i_{l-3}}\right)\nnm\\
&+(\iota+F_*)*\sum^{l-1}_{j=1}\left(\sum_{\sum(k+1)i_k=l-1-j}\prod^{l-2-j}_{k=0}(\td\nabla^k\td  T)^{i_k}\right) *\left(\sum_{\sum(k+1)i_k=j}\prod^{j-1}_{k=0}(\nabla^kh)^{i_k}\right),
\end{align}
where $\iota:{\mathcal N}\hookrightarrow F^*T\bbr^{m+p}$ is the inclusion map.
\end{rmk}

\begin{lem}\label{lem5.10}
If $|h|^2$ is bounded, then for each $l\geq 0$, there is a constant $C(l)$, such that $|\td\nabla^l\td  T|^2\leq C(l)$.
\end{lem}

\begin{proof}  We shall need the following formula:
\begin{align}
\td\nabla_t\td\nabla^{l} \td  T=&\sum_{r_0+r_1+r_2=l+1}\nabla^{r_0}a*\nabla^{r_1}h*\nabla^{r_2}h\nnm\\
&+\sum_{p,q\geq1;2\leq p+q\leq l+1}(*^p_{l+1-p-q}\td  T)*\sum_{r_0+r_1+r_2=q}\nabla^{r_0}a*\nabla^{r_1}h*\nabla^{r_2}h,\quad l\geq 0.\label{tdnablat t}
\end{align}

First we consider $l=0$. By the definition of the connection $\td\nabla$ on $M\times [0,T)$, it is easily seen that $\td\nabla_t\td T^k_{ij}=\pp{}{t}\td  T^k_{ij}=-\pp{}{t}\Gamma^k_{ij}$. Since for each fixed $t\in [0,T)$, $\pp{}{t}\Gamma^k_{ij}$ is a $(1,2)$-tensor on $M\times \{t\}$ w.r.t. the indices $i,j,k$, we can find by using normal coordinates $(x^i)$ and \eqref{evogij} that, on $M\times \{t\}$,
\begin{align}
\td\nabla_t\td\nabla^0\td T^k_{ij}=&-\pp{}{t}\Gamma^k_{ij}=-\fr12g^{kl}\left(\pp{}{x^j}\pp{g_{il}}{t}+\pp{}{x^i}\pp{g_{jl}}{t}-\pp{}{x^l}\pp{g_{ij}}{t}\right)\nnm\\
=&g^{kl}\left(\nabla_{\pp{}{x^j}}(aH^\alpha h^\alpha_{il}) +\nabla_{\pp{}{x^i}}(aH^\alpha h^\alpha_{jl})-\nabla_{\pp{}{x^l}}(aH^\alpha h^\alpha_{ij})\right)\nnm\\
=&\nabla a*h^2+a*h*\nabla h \label{nabla0t}\\
=&\sum_{r_0+r_1+r_2=0+1}\nabla^{r_0}a*\nabla^{r_1}h*\nabla^{r_2}h\nnm\\
&+\sum_{p,q\geq1;2\leq p+q\leq 0+1}(*^p_{0+1-p-q}\td  T)*\sum_{r_0+r_1+r_2=q}\nabla^{r_0}a*\nabla^{r_1}h*\nabla^{r_2}h,\nnm
\end{align}
that is, \eqref{tdnablat t} holds for $l=0$.

Suppose that \eqref{tdnablat t} holds for $l=k$. Then when $l=k+1$, we compute
\begin{align*}
\td\nabla_t\td\nabla^{k+1} \td  T=&\td\nabla(\td\nabla_t\td\nabla^k\td  T)=(\nabla+\td  T)\left(\sum_{r_0+r_1+r_2=k+1}\nabla^{r_0}a*\nabla^{r_1}h*\nabla^{r_2}h\right)\nnm\\
&+(\nabla+\td  T)\sum_{p,q\geq1;2\leq p+q\leq k+1}(*^p_{k+1-p-q}\td  T)*\sum_{r_0+r_1+r_2=q}\nabla^{r_0}a*\nabla^{r_1}h*\nabla^{r_2}h\\
=&\sum_{r_0+r_1+r_2=k+1}\nabla(\nabla^{r_0}a*\nabla^{r_1}h*\nabla^{r_2}h)+\td  T\sum_{r_0+r_1+r_2=k+1}\nabla^{r_0}a*\nabla^{r_1}h*\nabla^{r_2}h\nnm\\
&+\sum_{p,q\geq1;2\leq p+q\leq k+1}\nabla(*^p_{k+1-p-q}\td  T)*\sum_{r_0+r_1+r_2=q}\nabla^{r_0}a*\nabla^{r_1}h*\nabla^{r_2}h\\
&+\sum_{p,q\geq1;2\leq p+q\leq k+1}(*^p_{k+1-p-q}\td  T)*\sum_{r_0+r_1+r_2=q}\nabla(\nabla^{r_0}a*\nabla^{r_1}h*\nabla^{r_2}h) \\
&+\td T\sum_{p,q\geq1;2\leq p+q\leq k+1}(*^p_{k+1-p-q}\td  T)*\sum_{r_0+r_1+r_2=q}\nabla^{r_0}a*\nabla^{r_1}h*\nabla^{r_2}h\\
=&\sum_{r_0+r_1+r_2=k+2}\nabla^{r_0}a*\nabla^{r_1}h*\nabla^{r_2}h+\td  T\sum_{r_0+r_1+r_2=k+1}\nabla^{r_0}a*\nabla^{r_1}h*\nabla^{r_2}h\nnm\\
&+\sum_{p,q\geq1;2\leq p+q\leq k+1}(\td\nabla-\td T)(*^p_{k+1-p-q}\td  T)*\sum_{r_0+r_1+r_2=q}\nabla^{r_0}a*\nabla^{r_1}h*\nabla^{r_2}h\\
&+\sum_{p,q\geq1;2\leq p+q\leq k+1}(*^p_{k+1-p-q}\td  T)*\sum_{r_0+r_1+r_2=q+1}\nabla^{r_0}a*\nabla^{r_1}h*\nabla^{r_2}h \\
&+\sum_{p,q\geq1;2\leq p+q\leq k+1}(*^{p+1}_{k+1-p-q}\td  T)*\sum_{r_0+r_1+r_2=q}\nabla^{r_0}a*\nabla^{r_1}h*\nabla^{r_2}h\\
=&\sum_{r_0+r_1+r_2=k+2}\nabla^{r_0}a*\nabla^{r_1}h*\nabla^{r_2}h+\td  T\sum_{r_0+r_1+r_2=k+1}\nabla^{r_0}a*\nabla^{r_1}h*\nabla^{r_2}h\nnm\\
&+\sum_{p,q\geq1;2\leq p+q\leq k+1}\td\nabla(*^p_{k+1-p-q}\td  T)*\sum_{r_0+r_1+r_2=q}\nabla^{r_0}a*\nabla^{r_1}h*\nabla^{r_2}h\\
&-\sum_{p,q\geq1;2\leq p+q\leq k+1}\td T(*^p_{k+1-p-q}\td  T)*\sum_{r_0+r_1+r_2=q}\nabla^{r_0}a*\nabla^{r_1}h*\nabla^{r_2}h\\
&+\sum_{p,q\geq1;2\leq p+q\leq k+1}(*^p_{k+1-p-q}\td  T)*\sum_{r_0+r_1+r_2=q+1}\nabla^{r_0}a*\nabla^{r_1}h*\nabla^{r_2}h \\
&+\sum_{p,q\geq1;2\leq p+q\leq k+1}(*^{p+1}_{k+1-p-q}\td  T)*\sum_{r_0+r_1+r_2=q}\nabla^{r_0}a*\nabla^{r_1}h*\nabla^{r_2}h\\
=&\sum_{r_0+r_1+r_2=k+2}\nabla^{r_0}a*\nabla^{r_1}h*\nabla^{r_2}h+(*^1_0\,\td  T)*\sum_{r_0+r_1+r_2=k+1}\nabla^{r_0}a*\nabla^{r_1}h*\nabla^{r_2}h\nnm\\
&+\sum_{p,q\geq1;2\leq p+q\leq k+1}(*^p_{k+2-p-q}\td  T)*\sum_{r_0+r_1+r_2=q}\nabla^{r_0}a*\nabla^{r_1}h*\nabla^{r_2}h\\
&+\sum_{p\geq1,\td q\geq 2;3\leq p+\td q\leq k+2}(*^p_{k+2-p-\td q}\td  T)*\sum_{r_0+r_1+r_2=\td q}\nabla^{r_0}a*\nabla^{r_1}h*\nabla^{r_2}h\\
&+\sum_{\td p\geq 2,q\geq1;3\leq \td p+q\leq k+2}(*^{\td p}_{k+2-\td p-q}\td  T)*\sum_{r_0+r_1+r_2=q}\nabla^{r_0}a*\nabla^{r_1}h*\nabla^{r_2}h\\
=&\sum_{r_0+r_1+r_2=k+1+1}\nabla^{r_0}a*\nabla^{r_1}h*\nabla^{r_2}h\nnm\\
&+\sum_{p,q\geq1;2\leq p+q\leq k+1+1}(*^p_{k+1+1-p-q}\td  T)*\sum_{r_0+r_1+r_2=q}\nabla^{r_0}a*\nabla^{r_1}h*\nabla^{r_2}h.
\end{align*}
Therefore, \eqref{tdnablat t} holds for $l=k+1$ and thus for all $l\geq 0$ by induction.

Now we use \eqref{tdnablat t} to complete the proof of Lemma \ref{lem5.10}, using once more the induction.

Note that by Corollary \ref{cor5.8}, if $|h|^2$ is uniformly bounded from above for $t\in [0,T)$, then all of $\nabla^lh$, $\nabla^la$ and $\nabla^{l+1}F$ are also uniformly bounded for $l\geq 0$. On the other hand, by \eqref{gamat}, we know that $\nabla_t$ and $\td\nabla_t$ are related by
\be\label{ad5.25}
\nabla_tS=\td\nabla_tS +ah^2*S
\ee
for any bundle-valued tensor $S$ on $M\times[0,T)$. It then follows \eqref{nabla0t} and Young's inequality that, when $l=0$,
$$
\pp{}{t}|\td  T|^2=2\lagl \td  T,\td\nabla_t\td  T+ah^2*\td T\ragl\leq |\td  T*(\nabla a*h^2+ah*\nabla h+ah^2*\td T)|\leq C_1(1+|\td  T|^2),
$$
giving that
$$\pp{}{t}\log(1+|\td  T|^2)\leq C_1,\quad\text{or}\quad\pp{}{t}(\log(1+|\td  T|^2)-C_1t)\leq 0.$$
It follows from the compactness of $M$ that
$$
\log(1+|\td  T|^2)-C_1t\leq (\log(1+|\td  T|^2)-C_1t)|_{t=0}=(\log(1+|\td  T|^2))|_{t=0}\leq C_2.
$$
So $|\td  T|^2<1+|\td  T|^2\leq e^{C_1t+C_2}\leq C(0)$ since the interval $[0,T)$ is bounded.

Suppose that $|\td T|^2$, $\cdots$, $|\td\nabla^{l-1}\td  T|^2$ have been shown bounded for $l\geq 1$, then by \eqref{tdnablat t}, \eqref{ad5.25} and Young's inequality,
\begin{align*}
\pp{}{t}|\td\nabla^l\td T|^2=&2\lagl\td\nabla^l\td T,\td\nabla_t\td\nabla^l\td T+ah^2*\td\nabla^l\td T\ragl\\
\leq&\left|\td\nabla^l\td T*\Big(\sum_{r_0+r_1+r_2=l+1} \nabla^{r_0}a*\nabla^{r_1}h*\nabla^{r_2}h\right.\nnm\\
&\left.+\sum_{p,q\geq1;2\leq p+q\leq l+1}(*^p_{l+1-p-q}\td  T)*\sum_{r_0+r_1+r_2=q} \nabla^{r_0}a*\nabla^{r_1}h*\nabla^{r_2}h+ah^2*\td\nabla^l\td T\Big)\right|\\
\leq& C_1(1+|\td\nabla^l\td T|^2).
\end{align*}
Thus, as in the case of $l=0$, there exists some $C(l)>0$ such that
$|\td\nabla^l\td T|^2\leq C(l)$.

Now the principle of induction finishes the proof.
\end{proof}

Combing Corollary \ref{cor5.8}, Lemma \ref{lem5.9}, Lemma \ref{lem5.10} and the comparability  \eqref{32} of $g$ with $\td g$ we have proved the following boundedness result:

\begin{prop}\label{prop5.11} If $|h|^2$ is bounded as $t\to T$, then for each $l\geq 0$, there exists some constant $C(l)>0$ such that
$|\td\nabla^l F|^2_{\td g}\leq \td C(l)|\td\nabla^l F|^2_g\leq C(l)$ where $\td C(l)>0$.
\end{prop}

To prove Theorem \ref{thm5.1}, we also need the following formula:
\begin{lem} It holds that
\be\label{5.25}
\pp{}{t}\td\nabla^lF=\sum_{p+q=l}\left(\sum_{r+s+t=p}(*^r_s\td T)*\nabla^t a\right)*\left(\sum_{r+s+t=q+2,\,t\geq 1}(*^r_s\td T)*\td\nabla^t F\right),\quad l\geq 0.
\ee
\end{lem}

\begin{proof} Take $H$ as a section of the induced bundle $F^*T\bbr^{m+p}$. Since
$$
\pp{}{t}\td\nabla^lF=\td\nabla^l\pp{F}{t}=\td\nabla^l(aH)=\sum_{p+q=l}\td\nabla^pa*\td\nabla^qH, \quad l\geq 0,
$$
we only need to prove
\begin{align}
\td\nabla^pa=&\sum_{r+s+t=p}(*^r_s\td T)*\nabla^t a,\quad p\geq 0,\label{5.26}\\
\td\nabla^qH=&\sum_{r+s+t=q+2,\,t\geq1}(*^r_s\td T)*\td\nabla^t F,\quad q\geq 0.\label{5.27}
\end{align}
To do this, we shall use the method of induction as follows:

Firstly, it is easy to see that these two formulas are true for $p=q=0$.

Secondly, suppose that both \eqref{5.26} and \eqref{5.27} are true for $p\geq 0$ and $q\geq 0$, respectively. Then we find
\begin{align*}
\td\nabla^{p+1}a=&\td\nabla\left(\sum_{r+s+t=p}(*^r_s\td T)*\nabla^t a\right) =(\nabla+\td T)\left(\sum_{r+s+t=p}(*^r_s\td T)*\nabla^t a\right)\\
=&\sum_{r+s+t=p}\nabla(*^r_s\td T)*\nabla^t a+\sum_{r+s+t=p}(*^r_s\td T)*\nabla^{t+1} a +\sum_{r+s+t=p}\td T*(*^r_s\td T)*\nabla^t a\\
=&\sum_{r+s+t=p}(*^r_{s+1}\td T)*\nabla^t a+\sum_{r+s+t'=p+1,\,t'\geq1}(*^r_s\td T)*\nabla^{t'} a +\sum_{r+s+t=p}(*^{r+1}_s\td T)*\nabla^t a\\
=&\sum_{r+s'+t=p+1,\,s'\geq1}(*^r_{s'}\td T)*\nabla^t a+\sum_{r+s+t'=p+1,\,t'\geq1}(*^r_s\td T)*\nabla^{t'} a\\
& +\sum_{r'+s+t=p+1,\,r'\geq1}(*^{r'}_s\td T)*\nabla^t a\\
=&\sum_{r+s+t=p+1}(*^r_s\td T)*\nabla^t a,
\end{align*}
and
\begin{align*}
\td\nabla^{q+1}H=&\td\nabla\left(\sum_{r+s+t=q+2,\,t\geq1}(*^r_s\td T)*\td\nabla^t F\right)\\
=&(\nabla+\td T)\left(\sum_{r+s+t=q+2,\,t\geq1}(*^r_s\td T)*\td\nabla^t F\right)\\
=&\sum_{r+s+t=q+2,\,t\geq1}\nabla(*^r_s\td T)*\td\nabla^t F +\sum_{r+s+t=q+2,\,t\geq1}(*^r_s\td T)*\nabla\td\nabla^t F\\
&+\td T\sum_{r+s+t=q+2,\,t\geq1}(*^r_s\td T)*\td\nabla^t F\\
=&\sum_{r+s+t=q+2,\,t\geq1}(\td\nabla-\td T)(*^r_s\td T)*\td\nabla^t F +\sum_{r+s+t=q+2,\,t\geq1}(*^r_s\td T)*(\td\nabla-\td T)\td\nabla^t F\\
&+\sum_{r+s+t=q+2,\,t\geq1}(*^{r+1}_s\td T)*\td\nabla^t F\\
=&\sum_{r+s+t=q+2,\,t\geq1}\td\nabla(*^r_s\td T)*\td\nabla^t F -\sum_{r+s+t=q+2,\,t\geq1}\td T(*^r_s\td T)*\td\nabla^t F\\
&+\sum_{r+s+t=q+2,\,t\geq1}(*^r_s\td T)*\td\nabla\td\nabla^t F -\sum_{r+s+t=q+2,\,t\geq1}(*^r_s\td T)*\td T\td\nabla^t F\\
&+\sum_{r'+s+t=q+3,\,r'\geq 1,\,t\geq1}(*^{r'}_s\td T)*\td\nabla^t F\\
=&\sum_{r+s+t=q+2,\,t\geq1}(*^r_{s+1}\td T)*\td\nabla^t F +\sum_{r+s+t=q+2,\,t\geq1}(*^{r+1}_s\td T)*\td\nabla^t F\\
&+\sum_{r+s+t=q+2,\,t\geq1}(*^r_s\td T)*\td\nabla^{t+1} F+\sum_{r'+s+t=q+3,\,r'\geq 1,\,t\geq1}(*^{r'}_s\td T)*\td\nabla^t F\\
=&\sum_{r+s'+t=q+3,\,s'\geq1\,t\geq1}(*^r_{s'}\td T)*\td\nabla^t F +\sum_{r+s+t'=q+3,\,t'\geq2}(*^r_s\td T)*\td\nabla^{t'}F \\ &+\sum_{r'+s+t=q+3,\,r'\geq 1,\,t\geq1}(*^{r'}_s\td T)*\td\nabla^t F\\
=&\sum_{r+s+t=q+3,\,t\geq1}(*^r_s\td T)*\td\nabla^t F.
\end{align*}
Thus \eqref{5.26} and \eqref{5.27} are proved.
\end{proof}

{\em The proof of Theorem \ref{thm5.1}}.

From \eqref{5.19}, Lemma \ref{lem5.10}, Proposition \ref{prop5.11} and \eqref{5.25}, it is easily seen that, if Theorem \ref{thm5.1} were not true, then for any $l\geq 0$, $|\pp{}{t}\td\nabla^lF|$ would be uniformly bounded from above by a constant $C(l)>0$. It follows that
$$\left|\td\nabla^lF(x,t_1)-\td\nabla^lF(x,t_2)\right|\leq\left|\int^{t_2}_{t_1}\pp{}{t}\td\nabla^lFdt\right|\leq C(l)|t_1-t_2|,\quad$$
for all $x\in M^m$ and $t_1,t_2\in (0,T)$. So
$\td\nabla^lF$ would converge uniformly as $t\to T$, implying that $F(\cdot,t)$ would converge in $C^\infty$-topology to a limit immersion $F(\cdot,T):M^m\to\bbr^{m+p}$. By the existence theorem (Theorem \ref{exiuni}) of short time solution, $T$ could not be the maximal time of solution, which is a contradiction that proves Theorem \ref{thm5.1}.

\section{A direct proof of Theorem \ref{main2}}\label{s7}

In this section, we shall alternatively make use of those formulas derived in the appendix to provide a direct proof of Theorem \ref{main2}. Due to the discussions of Section \ref{s3}, we only need to find a direct proof for the following theorem (without using Theorem \ref{thmab}):

\begin{thm}\label{thm7.1} If specifically $a=a(t)$ depends only on the parameter $t$, then the maximal solution (deleting the hat) $F:M\times[0,T)\to\bbr^{m+p}$ of the conformal flow \eqref{4} must be convergent to a round point, provided that the initial submanifold $F_0$ satisfies the condition in Theorem \ref{thmab}.
\end{thm}

To this end, it suffices to do the following:

(1) Prove that $F_t(M)$ is convergent to a point as $t\to T$.

\newcommand{\diam}{{\rm diam}}
For the proof of (1), it only needs to show that the diameter $\diam F_t(M)\to 0$ as $t\to T$. This is well done by C. Baker in \cite{b} (Section 4.6), using Corollary \ref{cor6.11} and the following two theorems:

\begin{thm}[O. Bonnet, \cite{p}]\label{bonnet} Let $M$ be a complete Riemannian manifold and suppose that $x\in M$ such that the sectional curvature $K$ of $M$ satisfies $K\geq K_{min}>0$ along all geodesics of length $\pi/\sqrt{K_{min}}$ from $x$. Then $M$ is compact and the diameter $\diam M\leq\pi/\sqrt{K_{min}}$.
\end{thm}

\begin{thm}[B-Y Chen, \cite{c}] Let $M^m$ be a submanifold of $\bbr^{m+p}$ with $m\geq 2$. Then at each point $p\in M^m$, the smallest sectional curvature $K_{min}$ of $M^m$ satisfies
$$
K_{min}(p)\geq\fr12\left(\fr1{m-1}|H|^2(p)-|h|^2(p)\right).
$$
\end{thm}

Specifically, in the argument of \cite{b}, the key lemma is

\begin{lem}[\cite{b}, Lemma 4.24]\label{lem7.4} It holds that $\lim\limits_{t\to T}|H|^2_{max}=\infty$, and $\lim\limits_{t\to T}\fr{|H|_{max}}{|H|_{min}}=1$.
\end{lem}

(2) Prove that the solution $\td F:=\psi(t)F$ of the corresponding re-scaled volume-preserving flow (i.e. the so-called normalized flow) must be convergent to a round sphere.

To this end, we only need to copy and carefully check the main argument given in \cite{b} (Section 4.7) being applied to the new conformal mean curvature flow. It turns out that, some of these can be done without any change, and others will be done by new arguments. So, presently, it will suffice for us to outline the former and make in detail the latter as follows:

\newcommand{\vol}{{\rm Vol}}
Rescale the solution $F=F(x,t)$ of \eqref{4} by a function $\psi=\psi(t)$: $\td F=\psi(t)F$ of which all the geometric quantities are denoted by adding a tilde, with the following
\be\label{7.2} \vol(\td F_t(M))=\int_MdV_{\td g(t)}\equiv \vol(\td F_0(M)).\ee
Then we have (cf. \cite{b}):
\begin{align}
&\td g=\psi^2g,\quad \td h=\psi h,\quad \td H=\psi^{-1}H,\quad |\td h|^2=\psi^{-2}|h|^2,\label{7.3}\\
&\td\nabla=\nabla,\quad\td\Delta=\psi^{-2}\Delta,\quad dV_{\td g(t)}=\psi^mdV_{g(t)},\label{7.4}\\
&\psi^{-1}\dd{\psi}{t}=\fr 1m\hbar :=\fr 1m\fr{a\int |H|^2dV_{g(t)}}{\int dV_{g(t)}}\label{7.5}
\end{align}
Change the time parameter $t$ to $\td t$ by
\be\label{7.6} \td t(t)=\int^t_0\psi^2(\tau)d\tau,\quad 0\leq t<T,\quad \td T=\td t(T).
\ee
So $\dd{\td t}{t}=\psi^2(t)$. Define
$$
\td a(\td t)=a(t(\td t)),\quad\td\hbar=\fr{\td a\int |\td H|^2dV_{\td g(\td t)}}{\int dV_{\td g(\td t)}}.
$$
Then $\td\hbar =\psi^{-2}\hbar$ and
\be\label{7.7} \pp{\td F}{\td t}=\td a\td H+\fr1m\td\hbar\td F,\quad \psi^{-1}\dd{\psi}{\td t}=\fr 1m\td \hbar\geq\fr1m\ul a|\td H|^2_{min},\quad 0\leq\td t<\td T.\ee

\begin{prop}[cf. \cite{b}, Proposition 4.26] The following estimates hold for the normalized flow \eqref{7.7}:
\be\label{7.8}
|\td h|^2\leq c|\td H|^2,\quad\fr{|\td H|^2_{min}}{|\td H|^2_{max}}\to 1\text{ as }\td t\to \td T,\quad\td K_{min}\geq\veps^2|\td H|^2,
\ee
where, with some small $t_0>0$ and $\td a_{t_0}>0$,
$$\veps^2:=\fr12\left(\fr1{m-1}-c+\td a_{t_0}\right)>0$$
for all $t\geq t_0$ (see Corollary \ref{cor6.2}).
\end{prop}

\begin{lem} [cf. \cite{b}, Lemma 4.27]\label{lem7.6} Suppose that $P$ and $Q$ depend on $g$ and $h$, and that $P$ satisfies $\pp{P}{t}=a\Delta P+Q$. If $P$ has ``degree'' $\alpha$, that is, $\td P=\psi^\alpha P$, then $Q$ has degree $(\alpha-2)$ and $\td P$ satisfies the normalized evolution equation
\be
\pp{\td P}{\td t}=\td a\td\Delta \td P+\td Q+\fr\alpha m\td\hbar\td P.
\ee
\end{lem}

\begin{prop}[cf. \cite{b}, Propositions 4.31 and 4.32]\label{prop7.7} There are $C_{max}$ and $C_{min}$ such that
\be\label{7.11} 0<C_{min}\leq |\td H|_{min}\leq |\td H|_{max}\leq C_{max}<+\infty.\ee
\end{prop}

\begin{proof}
By using the third inequality in \eqref{7.8}, Bishop-Gromov volume comparision theorem (\cite{g-h-l}), Bonnet's theorem (Theorem \ref{bonnet}) and the volume-preserving property \eqref{7.2}, it can be shown that $|\td H|_{min}$ is bounded from above. Then the second inequality of \eqref{7.11} comes directly from the fact that $\lim\limits_{\td t\to \td T}\fr{|\td H|_{min}}{|\td H|_{max}}=1$.

On the other hand, the fact that $|\td h|^2\leq c|\td H|^2$ and the second inequality of \eqref{7.11} shows that the second fundamental form $\td h$ is bounded from above, that is,
\be\label{7.12} |\td h|^2\leq C\text{ for some }C>0.\ee

Now the first inequality of \eqref{7.11} comes from the second one, \eqref{7.12}, the Gr\"uther volume comparison theorem (\cite{g-h-l}) and the Klingenberg Lemma (\cite{p}). For the detail of the proof, see \cite{b}.
\end{proof}

\begin{cor}\label{cr}
There are $\td t_0, C,\delta>0$ such that $\psi\geq Ce^{\delta\td t}$ for all $\td t\geq\td t_0$.
\end{cor}

\begin{proof} From \eqref{7.7} and \eqref{7.11} it follows that
$$
\dd{}{\td t}\log\psi=\psi^{-1}\dd{\psi}{\td t}\geq\fr1m\ul a C^2_{min}:=\delta>0
$$
which proves the corollary.
\end{proof}

\begin{prop}[cf. \cite{b}, Lemma 4.39]\label{prop7.13} There exist positive constants $C',\delta'$ such that
\be
|\td\sch|^2\leq C'e^{-\delta'\td t},\quad
\forall \td t\geq\td t_0
\ee
for some sufficient large $\td t_0>0$.
\end{prop}

\begin{proof}
By \eqref{7.3} and Corollary \ref{cr}, we have for some $0<\sigma<1$
\begin{align*}
|\td\sch|^2=&|\td h|^2-\fr1m|\td H|^2=\psi^{-2}(|h|^2-\fr1m|H|^2)=\psi^{-2}|\sch|^2\\
\leq&C_0\psi^{-2}|H|^{2(1-\sigma)}=C_0\psi^{-2\sigma}|\td H|^{2(1-\sigma)} \leq C_0C^{-2\sigma}C^{2(1-\sigma)}_{max}e^{-2\sigma\delta\td t}.
\end{align*}
Take $C'=C_0C^{-2\sigma}C^{2(1-\sigma)}_{max}$ and $\delta'=2\sigma\delta$.
\end{proof}

\begin{prop}[cf. \cite{b}, Propositions 4.33 and 4.34] It holds that
\be\int_0^T|H|^2_{max}(t)dt=+\infty,\quad \td T=+\infty.\ee
\end{prop}

\begin{proof}  Following the proof of Theorem 15.3 in \cite{ha82} with $R_{max}$ replaced by $|H|^2_{max}$ and using
$$
\pp{}{t}|H|^2\leq a\Delta |H|^2+2ac|H|^2_{max}|H|^2,
$$
we can obtain the first equality.

Furthermore, by definition, $$\hbar =\fr{a\int_M|H|^2dV_{g(t)}}{\int_MdV_{g(t)}}\geq\ul a|H|^2_{min}.$$ This together with \eqref{7.5}, \eqref{7.6} and
$$\lim\limits_{t\to T}\fr{|H|_{min}}{|H|_{max}}=1,\quad\int_0^T|H|^2_{max}dt=+\infty$$ implies that
$$
\int^{\td T}_0\td\hbar(\td t)d\td t=\int^T_0\hbar dt\geq\ul a\int^T_0|H|^2_{min}dt=+\infty.
$$
However, since $\td\hbar\leq\ol a|\td H|^2_{max}\leq\ol aC^2_{max}<+\infty$, we have $\td T=+\infty$.
\end{proof}

Since we are now considering submanifolds of higher codimension, we need to extend the interpolation inequality of Hamilton for tensors into the following more general form:

\begin{lem}[Interpolation inequality for vector-valued tensors]\label{lem7.13} Let $E\to M^m$ be a Riemannian vector bundle on a compact Riemannian manifold $(M^m,g)$ with a metric connection $\nabla^E$ on $E$, and $r\geq 1$ be an integer. Suppose that $\fr1p+\fr1q=\fr1r$ and $T$ is a $E$-valued tensor on $M^m$. Then (\cite{ha82}, Theorem 12.1)
\be\label{ha82thm12.1}
\left(\int_M|\nabla T|^{2r}dV_g\right)^{\fr1r}\leq (2r-2+m)\left(\int_M|\nabla^2 T|^pdV_g\right)^{\fr1p}\left(\int_M|T|^qdV_g\right)^{\fr1q}.
\ee
Furthermore, for any $n\geq 1$,
there exists some constant $C$ depending only on $m,n$ such that (\cite{ha82}, Corollaries 12.6 and 12.7)
\begin{align}
\int_M|\nabla^iT|^{\fr{2n}{i}}dV_g\leq& C\max_{M^m}|T|^{2\left(\fr ni-1\right)}\int_M|\nabla^nT|^2dV_g,\quad 1\leq i\leq n-1;\label{ha82cor12.6}\\
\int_M|\nabla^iT|^2dV_g\leq& C\left(\int_M|\nabla^nT|^2dV_g\right)^{\fr in}\left(\int_M|T|^2dV_g\right)^{1-\fr in},\quad 0\leq i\leq n.\label{ha82cor12.7}
\end{align}
\end{lem}

\begin{proof} The argument of Section 12 in \cite{ha82} still applies here line by line.\end{proof}

\begin{prop}\label{prop7.12} For $l\geq 1$, there exist positive constants $C_l$ such that
$|\td\nabla^l\td h|^2\leq C_l$.
\end{prop}

\begin{proof} From \eqref{evohiderh} and the assumption we have for $l\geq 1$
\be\label{7.19}
\pp{}{t}|\nabla^lh|^2=a\big(\Delta|\nabla^lh|^2-2|\nabla^{l+1}h|^2 +\sum_{r_1+r_2+r_3=l}\nabla^{r_1}h*\nabla^{r_2}h*\nabla^{r_3}h*\nabla^lh\big).
\ee
Then by the proof of Theorem 7.3 in \cite{hui84}, using \eqref{ha82cor12.7}, we obtain
\begin{align*}
&\pp{}{t}\int_M|\nabla^lh|^2dV_{g(t)}+2\ul a\int_M|\nabla^{l+1}h|^2dV_{g(t)}\\ \leq&\pp{}{t}\int_M|\nabla^lh|^2dV_{g(t)}+2\ol a\int_M|\nabla^{l+1}h|^2dV_{g(t)}\\
\leq& C(l,m)\cdot\max|h|^2\int_M|\nabla^lh|^2dV_{g(t)},\quad l\geq 1.
\end{align*}
Making use of \eqref{7.3} and \eqref{7.4}, it follows that
\begin{align}
&\pp{}{\td t}\int_M|\td \nabla^l\td h|^2dV_{\td g(\td t)}+2\ul a\int_M|\td \nabla^{l+1}\td h|^2dV_{\td g(\td t)}\nnm\\
\leq&C(l,m)\cdot\max|\td h|^2\int_M|\td \nabla^l\td h|^2dV_{\td g(\td t)}\nnm\\
\leq&\td C\int_M|\td \nabla^l\td h|^2dV_{\td g(\td t)}\nnm\\
=&-\td C\int_M|\td \nabla^l\td h|^2dV_{\td g(\td t)} +2\td C\int_M|\td \nabla^l\td h|^2dV_{\td g(\td t)},\quad l\geq 1,\label{7.17}
\end{align}
since $\td h$ is bounded from above.

On the other hand, putting $n=l+1$, $i=l$ and $T=\td h$ in the interpolation \eqref{ha82cor12.7} we get by the Young inequality
\begin{align}
\int_M|\td \nabla^l\td h|^2dV_{\td g(\td t)}\leq& C\left(\int_M|\td \nabla^{l+1}\td h|^2dV_{\td g(\td t)}\right)^{\fr l{l+1}}\left(\int_M|\td h|^2dV_{\td g(\td t)}\right)^{1-\fr l{l+1}}\nnm\\
\leq& \fr l{l+1}\veps^{1+\fr1l}C_1\int_M|\td \nabla^{l+1}\td h|^2dV_{\td g(\td t)}+\fr1{(l+1)\veps^{l+1}}C_2,\quad l\geq 1.\label{7.18}
\end{align}
Choose $\veps$ small enough such that $$\fr{l}{l+1}\veps^{1+\fr1l}\td CC_1\leq \ul a.$$
Then \eqref{7.17} and \eqref{7.18} give that
$$
\pp{}{\td t}\int_M|\td \nabla^l\td h|^2dV_{\td g(\td t)}\leq -\td C\int_M|\td \nabla^l\td h|^2dV_{\td g(\td t)}+C_3
$$
or equivalently
$$
\pp{}{\td t}\left(\int_M|\td \nabla^l\td h|^2dV_{\td g(\td t)}-\fr{C_3}{\td C}\right)\leq -\td C \left(\int_M|\td \nabla^l\td h|^2dV_{\td g(\td t)}-\fr{C_3}{\td C}\right).
$$
There are here two cases that need to be considered:

Case (1) There exists some $\td t_0>0$, such that $\int_M|\td \nabla^l\td h|^2dV_{\td g(\td t)}\leq \fr{C_3}{\td C}$ for all $\td t>\td t_0$. In this case we can easily use the compactness of $[0,\td t_0]$ to conclude that
\be\label{l}
\int_M|\td \nabla^l\td h|^2dV_{\td g(\td t)}\leq C_l\quad\text{ with some } C_l>0;
\ee

Case (2) There exists an sequence
$$0<\td t_1<\td t'_1<\td t_2<\td t'_2<\cdots<\td t_\iota<\td t'_\iota\leq+\infty$$
that may be either infinite or finite, such that (choosing $C_3$ large enough)
\begin{align*}
&\int_M|\td \nabla^l\td h|^2dV_{\td g(\td t)}\leq \fr{C_3}{\td C}\quad\text{ on }I_1:=[0,\td t_1]\cup [\td t'_1,\td t_2]\cup\cdots\cup [\td t'_{\iota-1},\td t_\iota];\\
&\int_M|\td \nabla^l\td h|^2dV_{\td g(\td t)}>\fr{C_3}{\td C}\quad\text{ in }I_2:=(\td t_1,\td t'_1)\cup (\td t_2,\td t'_2)\cup\cdots\cup (\td t_\iota,\td t'_\iota).
\end{align*}
Therefore, we have
$$
\int_M|\td \nabla^l\td h|^2dV_{\td g(\td t)}-\fr{C_3}{\td C}\leq Ce^{-\td C\td t}\text{ for }\td t\in I_2.
$$
It then easily follows that there exists some $C_l>0$ such that \eqref{l} holds.

Now from Lemma \ref{lem7.13} we know that, for large $p$, the $L^p$-norm of $\td \nabla^l\td h$ is also uniformly bounded. Then by a suitable Sobolev inequality and the standard iteration, we can show that $|\td \nabla^l\td h|^2$ is uniformly bounded, proving the proposition.
\end{proof}

\begin{rmk}\rm We are to provide, in this remark, a new proof for Proposition \ref{prop7.12} without using neither Sobolev inequality nor the standard iteration, the detail of which is as follows:

First of all, we note that $|\td h|^2\leq C_0$ for some $C_0>0$. So we suppose that $|\td\nabla^i\td h|^2\leq C_i$, $C_i>0$, for $0\leq i\leq l-1$. Then, by \eqref{7.19}, Lemma \ref{7.6} and Proposition \ref{prop7.7}, we can directly write the evolution formula of $|\td\nabla^l\td h|^2$ as:
\begin{align}
\pp{}{\td t}|\td \nabla^l\td h|^2=&a\big(\td \Delta|\td \nabla^l\td h|^2-2|\td \nabla^{l+1}\td h|^2 +\sum_{r_1+r_2+r_3=l}\td \nabla^{r_1}\td h*\td \nabla^{r_2}\td h*\td \nabla^{r_3}\td h*\td \nabla^l\td h\big)\nnm\\
&-\fr{2(l+1)}{m}\td\hbar|\td\nabla^l\td h|^2\nnm\\
\leq& a \td \Delta|\td \nabla^l\td h|^2-2\ul a|\td \nabla^{l+1}\td h|^2+C'_l(1+|\td\nabla^l\td h|^2)\label{7.19'}
\end{align}
for some large $C'_l>0$. Let
$$f:=\fr{\td t}{\td t+1}|\td\nabla^l\td h|^2+\fr{C'_l}{\ul a}|\td\nabla^{l-1}\td h|^2.$$
Then we compute
\begin{align*}
\pp{f}{\td t}=&\left(\fr1{\td t+1}\right)^2|\td\nabla^l\td h|^2+\fr{\td t}{\td t+1}\pp{}{\td t}|\td\nabla^l\td h|^2+\fr{C'_l}{\ul a}\pp{}{\td t}|\td\nabla^{l-1}\td h|^2\\
\leq&\left(\fr1{\td t+1}\right)^2|\td\nabla^l\td h|^2+\fr{\td t}{\td t+1}
\left(a \td \Delta|\td \nabla^l\td h|^2-2\ul a|\td \nabla^{l+1}\td h|^2+C'_l(1+|\td\nabla^l\td h|^2)\right)\\
&+\fr{C'_l}{\ul a}\left(a \td \Delta|\td \nabla^{l-1}\td h|^2-2\ul a|\td \nabla^l\td h|^2+C''_l\right)\\
\leq&a \td \Delta f+\left(\left(\fr1{\td t+1}\right)^2+\fr{\td t}{\td t+1}C'_l-2C'_l\right)|\td \nabla^l\td h|^2+C'_l\left(\fr{\td t}{\td t+1}+\fr{C''_l}{\ul a}\right)\\
<&a \td \Delta f+\left(\fr1{\td t+1}\left(\fr1{\td t+1}-C'_l\right)-C'_l\right)|\td \nabla^l\td h|^2+C'_l\left(1+\fr{C''_l}{\ul a}\right)
\end{align*}
So when $C'_l\geq 1$,
\be\label{7.23}
\pp{f}{\td t}<a \td \Delta f-C'_l|\td \nabla^l\td h|^2+c'_l
\ee
where
$c'_l=C'_l\left(1+\fr{C''_l}{\ul a}\right)$.
Define
$$
U=\{(p,\td t)\in M\times[0,+\infty);\ C'_l|\td \nabla^l\td h|^2-c'_l>0\}.
$$

Case (1) $U$ is a empty set. Then we have $C'_l|\td \nabla^l\td h|^2-c'_l\leq 0$ everywhere and thus
$$|\td \nabla^l\td h|^2\leq C_l:=\fr{c'_l}{C'_l};$$

Case (2) $U$ is not empty. We claim that $f$ can not attain its maximal value on the closure $\ol U$ of $U$. In fact, if $f\leq f(p_0,\td t_0)$ for some $(p_0,\td t_0)\in\ol U$, then it must be that
$$
0=\pp{f}{\td t}(p_0,\td t_0)< -(C'_l|\td \nabla^l\td h|^2(p_0,\td t_0)-c'_l)\leq 0
$$
since $\td\Delta f(p_0,\td t_0)\leq 0$. This is of course not possible.

Now for each $\td t\in [0,+\infty)$, let $p_{\td t}\in M$ be the maximal value point of $f(\cdot,\td t)$. Then there is a $(p_0,\td t_0)\in M\times [0,+\infty]$ such that
$$
\lim\limits_{\td t\to \td t_0}f(p_{\td t},\td t)=\sup_{M\times[0,+\infty)} f.
$$
It is not hard to show that $\lim\limits_{\td t\to \td t_0}\pp{f}{\td t}(p_{\td t},\td t)\geq 0$. It follows from \eqref{7.23} that
$$\lim\limits_{\td t\to \td t_0}|\td \nabla^l\td h|^2(p_{\td t},\td t)\leq\fr{c'_l}{C'_l}.$$
Consequently, for all $\td t\geq 1$,
\begin{align*}
\fr12|\td\nabla^l\td h|^2\leq& f\leq\lim\limits_{\td t\to \td t_0}f(p_{\td t},\td t)
=\lim\limits_{\td t\to \td t_0}\left(\fr{\td t}{\td t+1}|\td\nabla^l\td h|^2(p_{\td t},\td t)\right)+\fr{C'_l}{\ul a}\lim\limits_{\td t\to \td t_0}|\td\nabla^{l-1}\td h|^2(p_{\td t},\td t)\\
\leq& \fr{c'_l}{C'_l}+\fr{C'_lC_{l-1}}{\ul a}.
\end{align*}
So, in Case (2), we also have the estimate:
$$|\td\nabla^l\td h|^2\leq C_l:=2\left(\fr{c'_l}{C'_l}+\fr{C'_lC_{l-1}}{\ul a}\right).$$
\end{rmk}

\begin{prop}[cf. \cite{b}, Proposition 4.40] The normalized submanifold $\td F_{\td t}(M)$ converges uniformly to a smooth limit submanifold $\td F_\infty(M)$ as $\td t\to +\infty$.
\end{prop}

\begin{proof} By using \eqref{7.4}, \eqref{nabla f1} and \eqref{nabla f2}, we find
\begin{align}
\td\nabla^{l+2}\td F=&\psi\Big(\nabla^l h+\sum_{\iota=0}^{k-1}(*^{2(k-\iota)}_{2\iota+1}h)^iF_*(e_i) +\sum_{\iota=0}^{k-1}(*^{2(k-\iota)+1}_{2\iota}h)^\alpha e_\alpha\Big),\quad \text{if }l=2k;\\
\td\nabla^{l+2}\td F=&\psi\Big(\nabla^l h+\sum_{\iota=0}^k(*^{2(k-\iota+1)}_{2\iota}h)^iF_*(e_i) +\sum_{\iota=0}^{k-1}(*^{2(k-\iota)+1}_{2\iota+1}h)^\alpha e_\alpha\Big),\quad \text{if }l=2k+1
\end{align}
where $k\geq 0$. It then follows from \eqref{7.3} and Proposition \ref{prop7.12} that, for $k\geq 0$ and $l=2k$,
\begin{align*}
|\td\nabla^{l+2}\td F|^2_{\td g}=&\psi^{-2(l+1)}|\nabla^{l+2}F|^2_g\\
=&\psi^{-2(l+1)}\Big(|\nabla^lh|^2_g +\sum_{\iota,\iota'=0}^{k-1}\lagl(*^{2(k-\iota)}_{2\iota+1}h), (*^{2(k-\iota')}_{2\iota'+1}h)\ragl_g\\ &+\sum_{\iota,\iota'=0}^{k-1}\lagl(*^{2(k-\iota)+1}_{2\iota}h), (*^{2(k-\iota')+1}_{2\iota'}h)\ragl_g +\sum_{\iota=0}^{k-1}\lagl\nabla^lh, (*^{2(k-\iota)+1}_{2\iota}h)\ragl_g\Big)\\
=&|\td\nabla^l\td h|^2_{\td g} +\sum_{\iota,\iota'=0}^{k-1}\lagl(*^{2(k-\iota)}_{2\iota+1}\td h), (*^{2(k-\iota')}_{2\iota'+1}\td h)\ragl_{\td g}\\ &+\sum_{\iota,\iota'=0}^{k-1}\lagl(*^{2(k-\iota)+1}_{2\iota}\td h), (*^{2(k-\iota')+1}_{2\iota'}\td h)\ragl_{\td g} +\sum_{\iota=0}^{k-1}\lagl\td\nabla^l\td h, (*^{2(k-\iota)+1}_{2\iota}\td h)\ragl_{\td g}\\
\leq&|\td\nabla^l\td h|^2_{\td g} +\sum_{\iota,\iota'=0}^{k-1}|*^{2(k-\iota)}_{2\iota+1}\td h|_{\td g} |*^{2(k-\iota')}_{2\iota'+1}\td h|_{\td g}\\ &+\sum_{\iota,\iota'=0}^{k-1}|*^{2(k-\iota)+1}_{2\iota}\td h|_{\td g} |*^{2(k-\iota')+1}_{2\iota'}\td h|_{\td g} +\sum_{\iota=0}^{k-1}|\td\nabla^l\td h|_{\td g} |*^{2(k-\iota)+1}_{2\iota}\td h|_{\td g}\\
\leq& \td C_l.
\end{align*}
Similarly, we have $|\td\nabla^{l+2}\td F|^2_{\td g}\leq \td C_l$ for $l=2k+1$, $k\geq 0$.

Now as did in the un-normalized case, we can replace the metric $\td g$ with an equivalent $\td t$-independent metric in the above estimate for the higher derivatives of $\td F$, from which the proposition follows easily.
\end{proof}

Finally, by Proposition \ref{prop7.13}, the limit submanifold $\td F_\infty(M)$ must be a compact and totally umbilic one in $\bbr^{m+p}$. Then an application of the Codazzi theorem (\cite{b}, \cite{sp}) leads to

\begin{prop}[cf. \cite{b}, Proposition 4.41] The limit submanifold $\td F_\infty(M)$ is an $m$-sphere lying in some $(m+1)$-dimensional subspace of $\bbr^{m+p}$.
\end{prop}

This last proposition completes the proof of Theorem \ref{thm7.1}.

\appendix
\section{More evolution formulas}

In this appendix, we are to derive some more formulas that, in our new situation, evolve a few important quantities introduced in \cite{a-b}. Note that these quantities and their involution formulas have played key roles in \cite{a-b} in proving the main convergence theorem (Theorem \ref{thmab}). We expect these computations will be of certain use in further study of the conformal flow in the Euclidean space. In particular, as is seen we have used those formulas derived in this appendix to alternatively give a direct proof of the convergence theorem (Theorem \ref{main2}) which has been proved already in Section \ref{s3} as the application of Theorem \ref{thm3.1} and the theorem of Andrews and Baker (Theorem \ref{thmab}).

As before, we denote $\sch:=h-\fr1mgH$. For any positive numbers $\td a$ and $c$, define as in \cite{a-b}
$Q=|h|^2+\td a-c|H|^2$. Then, by direct computation using \eqref{4}, we find
\begin{align}
\pp{Q}{t}-a\Delta Q=&-2a(|\nabla h|^2-c|\nabla H|^2)+2a(R_1-cR_2)\nnm\\
&+2(H^\alpha h^\alpha_{ij}a_{,kl}g^{ik}g^{jl}-c|H|^2\Delta a) +4(a_ih^\alpha_{kj}H^\alpha_{,l}g^{ik}g^{jl}-cH^\alpha H^\alpha_{,i}a_jg^{ij}).\label{qta}
\end{align}

When $c\leq\fr3{m+2}$ one has (see \cite{a-b}, Proposition 6) $|\nabla h|^2-c|\nabla H|^2\geq 0$. Consequently, we have

\begin{lem}\label{lem6.2} For $c\leq\fr3{m+2}$, it holds that
\begin{align}
&\pp{Q}{t}-a\Delta Q\nnm\\
\leq&2a(R_1-cR_2)+2(H^\alpha h^\alpha_{ij}a_{,ij}-c|H|^2\Delta a) +4(a_ih^\alpha_{ij}H^\alpha_{,j}-cH^\alpha H^\alpha_{,i}a_i).\label{qt}
\end{align}
\end{lem}

On the other hand, the authors of \cite{a-b} have also shown that
$$
\text{\em when $c\leq\fr4{3m}$, $R_1-cR_2$ is strictly negative at any point $(x,t)$ where $Q=0$.}\eqno (*)
$$
But by a careful examination of the argument for (*) we obtain

\begin{lem}\label{lem6.1} Let $c$ be as in \eqref{c}. Then,
at any point $(x,t)$ where $|h|^2\leq c|H|^2$, it holds that $R_1-cR_2\leq 0$.
\end{lem}

\begin{proof} It suffices to assume $H\neq 0$. Note that we always have the inequality $|h|^2\geq\fr1m|H|^2$. So $c\geq\fr1m$ since $H\neq 0$.

(1) If $c>\fr1m$, then a slight modification of the argument by \cite{a-b} in proving the statement (*) will be enough: just letting $\td a=0$ and using
$$-(c-\fr1m)|H|^2\leq-|\sch|^2 \text{ or }|H|^2\geq \fr1{c-\fr1m}|\sch|^2$$
instead will derive the following non-strict inequality
\begin{align*}
&2|\sch_1|^4-2(c-\fr2m)|\sch_1|^2|H|^2-\fr2m(c-\fr1m)|H|^4\\
\leq&2|\sch_1|^4-2|\sch_1|^2(|\sch_1|^2+|\sch_-|^2) -\fr2m|\sch|^2|H|^2+\fr2m|\sch_1|^2|H|^2\\
=&-2|\sch_1|\sch_-|^2-\fr2m|\sch_-|^2|H|^2\\
\leq&-\fr{2c}{c-\fr1m}|\sch_1|^2|\sch_-|^2 -\fr2{m(c-\fr1m)}|\sch_-|^4,
\end{align*}
and all other part of the argument does not need any change.

(2) If $c=\fr1m$, then we have $|h|^2-\fr1m|H|^2\equiv 0$, that is, $F_t(M)\subset\bbr^n$ is totally umbilic. So $h^\alpha_{ij}\equiv\fr1mH^\alpha g_{ij}$ for any $i,j,\alpha$. It then follows that $R_1-\fr1mR_2\equiv 0$.
\end{proof}

Then, by using Lemma \ref{lem6.2}, the argument in \cite{a-b} (see the proof of Theorem 2 and Proposition 7 there) presently applies, giving the following corollary:

\begin{cor}\label{cor6.2} Suppose that $F:M\times[0,T)\to \bbr^{m+p}$ is a solution of \eqref{4} and that $a\equiv a(t)$ depends only on the parameter $t$. If the initial submanifold $F_0$ is such that $|h|^2\leq c|H|^2\neq 0$ with a constant $c$ satisfying \eqref{c}, then for any small $t_0\in(0,T)$, there exist some $c\leq\fr4{3m}$ and a constant $\td a_{t_0}>0$ such that $|h|^2+\td a_{t_0}\leq c|H|^2$ on $[t_0,T)$. In particular, $|h|^2<c|H|^2\neq 0$ for all $t\in(0,T)$.
\end{cor}

Now suppose that $|H|^2\neq 0$ and denote
$$c_t:=\min_{p\in M}\{d:|h_t|^2\leq d|H_t|^2\},\quad t\in [0,T).$$
For a fixed positive number $\sigma<1$ small enough, define $f_\sigma:=|\sch|^2|H|^{2(\sigma-1)}$. Then we have
\begin{lem}[cf. \cite{a-b}, Proposition 10]\label{lem6.3} If $c_t$ meets \eqref{c} then,
for any $\sigma\in(0,\fr12]$ and $t\in[0,T)$, it holds that
\begin{align}
\pp{}{t}f_\sigma\leq&a\Delta f_\sigma-\fr{4a(\sigma-1)}{|H|}\lagl\nabla |H|,\nabla f_\sigma\ragl-2a\veps_\nabla|H|^{2(\sigma-1)}|\nabla H|^2+2a\sigma|h|^2f_\sigma\nnm\\
&+2|H|^{2(\sigma-1)}\left(H^\alpha h^\alpha_{ij}a_{,ij} +2a_ih^\alpha_{ij}H^\alpha_{,j}\right)\nnm\\
&-\fr2m|H|^{2(\sigma-2)}(|H|^2\Delta a+\lagl\nabla|H|^2,\nabla a\ragl)\left(|H|^2+m(1-\sigma)|\sch|^2\right).\label{evfsig}
\end{align}
where $\veps_\nabla\equiv\veps_\nabla(t): =\fr3{m+2}-c_t$.
\end{lem}

{\rmk\rm Since $c_t\geq\fr1m$, if define
$$
\veps_0: =\begin{cases}\fr{5m-8}{3m(m+2)},&2\leq m\leq 4,\\
\fr{2m-5)}{(m-1)(m+2)},&m\geq 4,
\end{cases}
$$
then
$$0<\veps_0\leq\veps_\nabla\leq\fr{2(m-1)}{m(m+2)}<\fr13.$$

\begin{proof} Following the proof of Proposition 10 in \cite{a-b}, we find that
\begin{align*}
\pp{}{t}f_\sigma=&\left(\pp{}{t}|h|^2 -\fr1m\pp{}{t}|H|^2\right)|H|^{2(\sigma-1)} +(\sigma-1)|\sch|^2 |H|^{2(\sigma-2)}\pp{}{t}|H|^2\\
=&\Big(a\Delta|h|^2-2a|\nabla h|^2+2aR_1+2H^\alpha h^\alpha_{ij}a_{,ij}+4a_ih^\alpha_{ij}H^\alpha_{,j}\\
&-\fr1m a\left(\Delta|H|^2-2|\nabla H|^2+2R_2\right) -\fr2m\left(|H|^2\Delta a+\lagl\nabla|H|^2,\nabla a\ragl\right)\Big)|H|^{2(\sigma-1)}\\
&+(\sigma-1)|H|^{2(\sigma-2)}|\sch|^2 \Big(a\left(\Delta|H|^2-2|\nabla H|^2+2R_2\right)\\
&+2\left(|H|^2\Delta a+\lagl\nabla|H|^2,\nabla a\ragl\right)\Big)\\
=&a|H|^{2(\sigma-1)}\Big(\Delta|h|^2-2|\nabla h|^2+2R_1-\fr1m\left(\Delta|H|^2-2|\nabla H|^2+2R_2\right)\\
&+(\sigma-1)|H|^{-2}|\sch|^2 \left(\Delta|H|^2-2|\nabla H|^2+2R_2\right)\Big)\\
&+2|H|^{2(\sigma-1)}\Big(H^\alpha h^\alpha_{ij}a_{,ij}+2a_ih^\alpha_{ij}H^\alpha_{,j}\Big)\\
& -\fr2m|H|^{2(\sigma-2)}\left(|H|^2\Delta a+\lagl\nabla|H|^2,\nabla a\ragl\right)\left(|H|^2
-m(\sigma-1)|\sch|^2\right).
\end{align*}
But it is known that (\cite{a-b})
\begin{align*}
\Delta f_\sigma=&\left(\Delta|h|^2 -\fr1m\Delta|H|^2\right)|H|^{2(\sigma-1)}\\
&+|\sch|^2\Big((\sigma-1)(\sigma-2)|H|^{2(\sigma-3)}|\nabla|H|^2|^2 +(\sigma-1)|H|^{2(\sigma-2)}\Delta|H|^2\Big)\\
&+2(\sigma-1)|H|^{2(\sigma-2)}\lagl\nabla|h|^2-\fr1m\nabla|H|^2,\nabla|H|^2\ragl\\
=&\left(\Delta|h|^2 -\fr1m\Delta|H|^2\right)|H|^{2(\sigma-1)}\\ &+|\sch|^2\Big((\sigma-1)(\sigma-2)|H|^{2(\sigma-3)}|\nabla|H|^2|^2 +(\sigma-1)|H|^{2(\sigma-2)}\Delta|H|^2\Big)\\
&+\fr{2(\sigma-1)}{|H|^2}\lagl\nabla f_\sigma,\nabla|H|^2\ragl -\fr{8(\sigma-1)^2}{|H|^2}f_\sigma|\nabla|H||^2.
\end{align*}
Comparing the above two equalities then gives
\begin{align*}
\pp{}{t}f_\sigma=&a\Big(\Delta f_\sigma+\fr{2(1-\sigma)}{|H|^2}\lagl\nabla f_\sigma,\nabla |H|^2\ragl-2|H|^{2(\sigma-1)}\Big(|\nabla h|^2-\fr{|h|^2}{|H|^2}|\nabla H|^2\Big)+\fr{2\sigma R_2f_\sigma}{|H|^2}\\
&+\fr{4\sigma(\sigma-1)}{|H|^2}f_\sigma|\nabla|H||^2-\fr{2\sigma|\sch|^2}{|H|^{2(2-\sigma)}}|\nabla H|^2+2|H|^{2(\sigma-1)}\Big(R_1-\fr{|h|^2}{|H|^2}R_2\Big)\Big)\\
&+2|H|^{2(\sigma-1)}\left(H^\alpha h^\alpha_{ij}a_{,ij} +2a_ih^\alpha_{ij}H^\alpha_{,j}\right)\nnm\\
& -\fr2m|H|^{2(\sigma-2)}\left(|H|^2\Delta a+\lagl\nabla|H|^2,\nabla a\ragl\right)\left(|H|^2
-m(\sigma-1)|\sch|^2\right)\\
\leq&a\Big(\Delta f_\sigma+\fr{2(1-\sigma)}{|H|^2}\lagl\nabla f_\sigma,\nabla |H|^2\ragl-2|H|^{2(\sigma-1)}\Big(|\nabla h|^2-\fr{|h|^2}{|H|^2}|\nabla H|^2\Big)+\fr{2\sigma R_2f_\sigma}{|H|^2}\Big)\\
&+2|H|^{2(\sigma-1)}\left(H^\alpha h^\alpha_{ij}a_{,ij} +2a_ih^\alpha_{ij}H^\alpha_{,j}\right)\nnm\\
& -\fr2m|H|^{2(\sigma-2)}\left(|H|^2\Delta a+\lagl\nabla|H|^2,\nabla a\ragl\right)\left(|H|^2
-m(\sigma-1)|\sch|^2\right),
\end{align*}
where in the last inequality we have used the fact that $R_1-\fr{|h|^2}{|H|^2}R_2\leq 0$ (see Lemma \ref{lem6.1}).
This together with the estimate
$$
-2|H|^{2(\sigma-1)}\left(|\nabla h|^2-\fr{|h|^2}{|H|^2}|\nabla H|^2\right)\leq -2|H|^{2(\sigma-1)}\left(\fr{3}{m+2}-c_t\right)|\nabla H|^2
$$
and \eqref{4.13} proves \eqref{evfsig}.
\end{proof}

\begin{lem}[cf. \cite{a-b}, Proposition 13]\label{lem6.4} Let $c_t$ be as in Lemma \ref{lem6.3}. Then it holds that,
for any $p\geq\max\{2,\fr{8\ol a^2}{\ul a^2\veps_\nabla}+1\}$,
\begin{align*}
\pp{}{t}\int_Mf^p_\sigma dV\leq&-\fr12\ul a p(p-1)\int_Mf^{p-2}_\sigma|\nabla f_\sigma|^2dV -\ul ap\veps_\nabla\int_Mf^{p-1}_\sigma|H|^{2(\sigma-1)}|\nabla H|^2dV\\
&+2\ol ap\sigma\int_M|H|^2f^p_\sigma dV -p\int_Mf^{p-1}_\sigma\lagl\nabla a,\nabla f_\sigma\ragl dV\\
&+2p\int_Mf^{p-1}_\sigma|H|^{2(\sigma-1)}\Big(H^\alpha h^\alpha_{ij}a_{,ij}+2a_ih^\alpha_{ij}H^\alpha_{,j}\Big)dV\\
&-\fr{2p}m\int_M\fr{f^{p-1}_\sigma}{|H|^{2(2-\sigma)}} \left(|H|^2\Delta a+\lagl\nabla|H|^2,\nabla a\ragl\right)\left(|H|^2
-m(\sigma-1)|\sch|^2\right)dV.
\end{align*}
\end{lem}

\begin{proof} Differentiating under the integral sign and substituting in the evolution equations for $f_\sigma$ and for the measure $dV$ gives
\begin{align*}
\pp{}{t}\int_Mf^p_\sigma dV=&\int_M\big(pf^{p-1}_\sigma\pp{f_\sigma}{t}-a|H|^2f^p_\sigma\big)dV\\
\leq&\int_Mpf^{p-1}_\sigma\pp{f_\sigma}{t}dV\\
\leq&p\int_Maf^{p-1}_\sigma\Delta f_\sigma dV+4(1-\sigma)p\int_Ma\fr{f^{p-1}_\sigma}{|H|}\lagl\nabla|H|,\nabla f_\sigma\ragl dV\\
&-2p\veps_\nabla\int_Ma\fr{f^{p-1}_\sigma}{|H|^{2(1-\sigma)}}|\nabla H|^2dV+2p\sigma\int_Ma|h|^2f^{p}_\sigma dV\\
&+2p\int_Mf^{p-1}|H|^{2(\sigma-1)}\Big(H^\alpha h^\alpha_{ij}a_{,ij}+2a_ih^\alpha_{ij}H^\alpha_{,j}\Big)dV\\
&-\fr{2p}m\int_M\fr{f^{p-1}_\sigma}{|H|^{2(2-\sigma)}} \big(|H|^2\Delta a+\lagl\nabla|H|^2,\nabla a\ragl\big)\big(|H|^2
-m(\sigma-1)|\sch|^2\big)dV.
\end{align*}
By using
$$p\int_Maf^{p-1}_\sigma\Delta f_\sigma dV\leq-p\int_Mf^{p-1}_\sigma\lagl\nabla a,\nabla f_\sigma\ragl dV-\ul a p(p-1)\int_Mf^{p-2}_\sigma|\nabla f_\sigma|^2dV$$
and the inequality (see \cite{a-b})\begin{align*}&4(1-\sigma)p\int_M\fr{f^{p-1}_\sigma}{|H|}\lagl\nabla|H|,\nabla f_\sigma\ragl dV\\
&\qquad\qquad\leq\fr{2p}{\rho}\int_Mf^{p-2}_\sigma|\nabla f_\sigma|^2dV+2p\rho\int_M\fr{f^{p-1}_\sigma}{|H|^{2(1-\sigma)}}|\nabla H|^2dV\end{align*}
in which the number $\rho>0$ is to be determined, we obtain that
\begin{align*}
&\pp{}{t}\int_Mf^p_\sigma dV\\
\leq&-p(p-1)(\ul a-\fr{2\ol a}{\rho(p-1)})\int_Mf^{p-2}_\sigma|\nabla f_\sigma|^2dV\\
&-2p\veps_\nabla(\ul a-\fr{\ol a\rho}{\veps_\nabla})\int_Mf^{p-1}_\sigma|H|^{2(\sigma-1)}|\nabla H|^2dV\\
&+2\ol ap\sigma\int_M|H|^2f^p_\sigma dV -p\int_Mf^{p-1}_\sigma\lagl\nabla a,\nabla f_\sigma\ragl dV\\
&+2p\int_Mf^{p-1}_\sigma|H|^{2(\sigma-1)}\Big(H^\alpha h^\alpha_{ij}a_{,ij}+2a_ih^\alpha_{ij}H^\alpha_{,j}\Big)dV\\
&-\fr{2p}m\int_M\fr{f^{p-1}_\sigma}{|H|^{2(2-\sigma)}} \left(|H|^2\Delta a+\lagl\nabla|H|^2,\nabla a\ragl\right)\left(|H|^2
-m(\sigma-1)|\sch|^2\right)dV,
\end{align*}
where we have used the fact that $|h|^2\leq c_t|H|^2\leq |H|^2$.
Choose
$$\rho=\fr{4\ol a}{\ul a(p-1)}\quad\text{and}\quad p\geq\max\{2,\fr{8\ol a^2}{\ul a^2\veps_\nabla}+1\},$$
then one has $\ul a-\fr{2\ol a}{\rho(p-1)}\geq \fr12\ul a$ and $\ul a-\fr{\ol a\rho}{\veps_\nabla}\geq \fr12\ul a$. This gives the desired inequality.
\end{proof}

\begin{prop}\label{prop6.5} Under the condition of Lemma \ref{lem6.4}, there exists constant a constant $\sigma_+$ and $c_9$ depending only on $M$ such that, if $\sigma\leq \sigma_+$, then for all $t\in[0,T)$ we have the estimate
\begin{align*}
\pp{}{t}\int_Mf^p_\sigma dV\leq&-p\int_Mf^{p-1}_\sigma\lagl\nabla a,\nabla f_\sigma\ragl dV
+2p\int_Mf^{p-1}_\sigma|H|^{2(\sigma-1)}\Big(H^\alpha h^\alpha_{ij}a_{,ij}+2a_ih^\alpha_{ij}H^\alpha_{,j}\Big)dV\\
&-\fr{2p}m\int_M\fr{f^{p-1}_\sigma}{|H|^{2(2-\sigma)}} \left(|H|^2\Delta a+\lagl\nabla|H|^2,\nabla a\ragl\right)\left(|H|^2
-m(\sigma-1)|\sch|^2\right)dV.
\end{align*}
\end{prop}

\begin{proof}
Since by \cite{a-b} (Proposition 12), for small $\sigma$, any $\eta>0$ and some $\veps>0$,
$$
\int_Mf^p_\sigma|H|^2dV\leq\fr{(p\eta+4)}{\veps} \int_M\fr{f^{p-1}_\sigma}{|H|^{2(1-\sigma)}}|\nabla H|^2dV +\fr{p-1}{\veps\eta}\int_Mf^{p-2}_\sigma|\nabla f_\sigma|^2dV,
$$
we find using Lemma \ref{lem6.4} that
\begin{align*}
\pp{}{t}\int_Mf^p_\sigma dV\leq&-\fr{p(p-1)\ul a}2\left(1-\fr{4\ol a}{\ul a}\cdot\fr{\sigma}{\veps\eta}\right) \int_Mf^{p-2}_\sigma|\nabla f_\sigma|^2dV \\ &-p\ul a\veps_\nabla\left(1-\fr{2\ol a}{\ul a}\cdot\fr{\sigma(p\eta+4)}{\veps\veps_\nabla}\right)\int_M\fr{f^{p-1}_\sigma} {|H|^{2(1-\sigma)}}|\nabla H|^2dV\\
&-p\int_Mf^{p-1}_\sigma\lagl\nabla a,\nabla f_\sigma\ragl dV
+2p\int_Mf^{p-1}_\sigma|H|^{2(\sigma-1)}\Big(H^\alpha h^\alpha_{ij}a_{,ij}+2a_ih^\alpha_{ij}H^\alpha_{,j}\Big)dV\\
&-\fr{2p}m\int_M\fr{f^{p-1}_\sigma}{|H|^{2(2-\sigma)}} \left(|H|^2\Delta a+\lagl\nabla|H|^2,\nabla a\ragl\right)\left(|H|^2
-m(\sigma-1)|\sch|^2\right)dV.
\end{align*}

\newcommand{\bara}{\fr{\ol a}{\ul a}}
Put $\eta=4\bara\cdot\fr\sigma\veps$. Then
\be\label{6.3}
1-\fr{4\ol a}{\ul a}\cdot\fr{\sigma}{\veps\eta}=0
\ee
and
\begin{align*}
\fr{2\ol a}{\ul a}\cdot\fr{\sigma(p\eta+4)}{\veps\veps_\nabla}-1
=&\fr1{\veps\veps_\nabla}\left(2\bara\cdot\sigma \left(4p\bara\cdot\fr\sigma\veps+4\right) -\veps\veps_\nabla\right)\\
=&\fr1{\veps^2\veps_\nabla}\left(8p\left(\bara\right)^2 \sigma^2+8\bara\veps\sigma-\veps^2\veps_\nabla\right)\\
\leq&\fr1{\veps^2\veps_\nabla} \left(8p\left(\bara\right)^2\sigma^2+8\bara\veps\sigma -\veps^2\veps_0\right).
\end{align*}
Denote
$$
f(\sigma):=8p\left(\bara\right)^2\sigma^2+8\bara\veps\sigma -\veps^2\veps_0.
$$
Then as a quadratic polynomial of $\sigma$, $f(\sigma)$ has two roots as follows:
\begin{align*}
&\sigma_-=-\fr\veps{2\bara p}\left(1+\sqrt{1+\fr12p\veps_0}\right)<0,\\
&\sigma_+=\fr\veps{2\bara p}\left(-1+\sqrt{1+\fr12p\veps_0}\right)=\fr{\ul a\veps\veps_\nabla}{4\ol a(1+\sqrt{1+\fr12p\veps_0})}>0.
\end{align*}
Consequently, for any $0<\sigma\leq\sigma_+$, we have $f(\sigma)\leq 0$. This with \eqref{6.3} completes the proof of Proposition \ref{prop6.5}.
\end{proof}

\begin{lem} If $|h|^2\leq c|H|^2$, then there is a constant $A$ depending on $F_0$ and the bound of $a$ such that
\begin{align}
\pp{}{t}|\nabla H|^2\leq& a(\Delta|\nabla H|^2-2|\nabla^2H|^2)+A|H|^2|\nabla h|^2\nnm\\
&+\nabla a*(\nabla h*\nabla^2h+h^3*\nabla h)+\nabla^2a*(\nabla h)^2+\nabla^3a*h*\nabla h.\label{ptnabH2}
\end{align}
\end{lem}

\begin{proof}
This lemma can be proved by a direct computation as follows:  By using the Ricci identity we find
\begin{align*}
&\Delta|\nabla H|^2=2\lagl\Delta\nabla_kH,\nabla_kH\ragl+2|\nabla^2H|^2\\
=&2\lagl\nabla_k\Delta H+\nabla_p(R^\bot(e_p,e_k)H)+R^\bot(e_p,e_k)\nabla_p H+\ric(e_p,e_k)\nabla_pH,\nabla_kH\ragl+2|\nabla^2H|^2\\
=&2\lagl\nabla_k\Delta H,\nabla_kH\ragl+2|\nabla^2H|^2+h^2*(\nabla h)^2.
\end{align*}
Thus
$$
2\lagl\nabla_k\Delta H,\nabla_kH\ragl=\Delta|\nabla H|^2-2|\nabla^2H|^2+h^2*(\nabla h)^2
$$
which implies that
\begin{align*}
2\lagl\nabla_k\nabla_tH,\nabla_k H\ragl=&2a\lagl\nabla_k\Delta H,\nabla_k H\ragl+a*h^2*(\nabla h)^2\\
&+\nabla a*(\nabla h*\nabla^2h+h^3*\nabla h)+\nabla^2a*(\nabla h)^2+\nabla^3a*h*\nabla h\\
=&a(\Delta|\nabla H|^2-2|\nabla^2H|^2)+a*h^2*(\nabla h)^2\\
&+\nabla a*(\nabla h*\nabla^2 h+h^3*\nabla h)+\nabla^2a*(\nabla h)^2+\nabla^3a*h*\nabla h.
\end{align*}
Therefore
\begin{align}
\pp{}{t}|\nabla H|^2=&2\lagl\nabla_t\nabla_k H,\nabla_k H\ragl=2\lagl \nabla_k\nabla_tH+R^\bot(\partial_t,e_k)H,\nabla_kH\ragl)\nnm\\
=&2\lagl\nabla_k\nabla_t H,\nabla_kH\ragl+2\lagl R^\bot(\partial t,e_k)H,\nabla_kH\ragl\nnm\\
=&a(\Delta|\nabla H|^2-2|\nabla^2H|^2+h^2*(\nabla h)^2)\nnm\\
&+\nabla a*(\nabla h*\nabla^2h+h^3*\nabla h)+\nabla^2a*(\nabla h)^2+\nabla^3a*h*\nabla h.\label{ptnabH}
\end{align}
Since $a$ is bounded, we have $a*h^2*(\nabla h)^2\leq A|H|^2|\nabla h|^2$ for some constant $A$. Inserting this into \eqref{ptnabH} gives \eqref{ptnabH2}
\end{proof}

\begin{lem} If $|h|^2\leq c|H|^2$ and there exist some positive constants $C_0$ and $\sigma<1$ such that $|\sch|^2\leq C_0|H|^{2(1-\sigma)}$, then for any constants $N_1, N_2>0$, it holds that
\begin{align}
&\pp{}{t}|H|^4\geq a(\Delta |H|^4-12|H|^2|\nabla H|^2+\fr4m|H|^6)+\nabla a*h^3*\nabla h+\nabla^2 a*h^4,\label{6.7}\\
&\pp{}{t}((N_1+N_2|H|^2)|\sch|^2)\leq a\Big(\Delta((N_1+N_2|H|^2)|\sch|^2)-\fr{4(m-1)}{3m}(N_2-1)|H|^2|\nabla h|^2\nnm\\
&\quad-\fr{4(m-1)}{3m}(N_1-c_1(N_2))|\nabla h|^2+c_2(N_1,N_2)|\sch|^2(|H|^4+1)\Big)\nnm\\
&\quad+\nabla a*(h^3*\nabla h+h*\nabla h)+\nabla^2a*(h^4+h^2)\label{6.8}
\end{align}
for some constants $c_1$ and $c_2$.
\end{lem}

\begin{proof} The evolution equation for $|H|^4$ is easily derived from that of $|H|^2$: First we find
\begin{align*}
\pp{}{t}|H|^4=&a(\Delta |H|^4-8|H|^2|\nabla |H||^2-4|H|^2|\nabla H|^2+4|H|^2R_2)\\
&+4|H|^2\lagl \nabla a ,\nabla|H|^2\ragl+4|H|^4\Delta a \\
=&a(\Delta |H|^4-8|H|^2|\nabla |H||^2-4|H|^2|\nabla H|^2+4|H|^2R_2)\\
&+\nabla a*h^3*\nabla h+\nabla^2 a*h^4.
\end{align*}
Then, from
\begin{align*}
R_2=&|\lagl H,h\ragl|^2=|\lagl H,\sch \ragl|^2+\fr1m|H|^4\geq\fr1m|H|^4
\end{align*}
and the classical Kato inequality $|\nabla |H||^2\leq |\nabla H|^2$, comes Equation \eqref{6.7} follows directly.

To prove \eqref{6.8}, we directly compute by the evolution equations for $|h|^2$ and $|H|^2$:
\begin{align*}
&\pp{}{t}((N_1+N_2|H|^2)|\sch|^2)\\
=&N_2\big(a(\Delta|H|^2-2|\nabla H|^2+2R_2)|\sch|^2+(\nabla a*h*\nabla h+\nabla^2a*h^2)|\sch|^2\big)\\
&+(N_1+N_2|H|^2)\Big(a\big(\Delta|\sch|^2-2(|\nabla h|^2-\fr1m|\nabla H|^2)+2(R_1-\fr1mR_2)\big)\\
&+\nabla a*h*\nabla h+\nabla^2a*h^2\Big).
\end{align*}
Since
\begin{align*}
&\Delta((N_1+N_2|H|^2)|\sch|^2)\\
=&N_2(\Delta|H|^2)|\sch|^2+(N_1+N_2|H|^2)\Delta|\sch|^2 +2N_2\lagl\nabla|H|^2,\nabla|\sch|^2\ragl,
\end{align*}
and (\cite{a-b})
\begin{align}
&|\nabla h|^2-\fr1m|\nabla H|^2\geq\fr{2(m-1)}{3m}|\nabla h|^2,\quad |\nabla h|^2\geq \fr1m|\nabla H|^2,\label{ad3}\\
&R_2\leq |h|^2|H|^2,\quad R_1-\fr1m R_2 \leq c_0|\sch|^2|H|^2,\label{ad4}
\end{align}
with $c_0$ being a constant dependent on $c$, we find
\begin{align}
&\pp{}{t}((N_1+N_2|H|^2)|\sch|^2)\nnm\\
=&a\Big(\Delta\big((N_1+N_2|H|^2)|\sch|^2\big) -2N_2\lagl\nabla|H|^2,\nabla|\sch|^2\ragl-2N_2|\nabla H|^2|\sch|^2+2N_2R_2|\sch|^2\nnm\\
&-2(N_1+N_2|H|^2)(|\nabla h|^2-\fr1m|\nabla H|^2)+2(N_1+N_2|H|^2)(R_1-\fr1mR_2)\Big)\nnm\\
&+\nabla a*(h^3*\nabla h+h*\nabla h)+\nabla^2a*(h^4+h^2)\label{ad1}\\
\leq&a\Big(\Delta\big((N_1+N_2|H|^2)|\sch|^2\big) +8N_2|\sch||H|\cdot|\nabla|H||\cdot|\nabla|\sch||\nnm\\
&-\fr{4(m-1)}{3m}(N_1+N_2|H|^2)|\nabla h|^2+2N_2|\sch|^2|h|^2|H|^2
+2c_0(N_1+N_2|H|^2)|\sch|^2|H|^2\Big)\nnm\\
&+\nabla a*(h^3*\nabla h+h*\nabla h)+\nabla^2a*(h^4+h^2)\nnm\\
\leq&a\Big(\Delta\big((N_1+N_2|H|^2)|\sch|^2\big)+8N_2|\sch||H|\cdot|\nabla H|\cdot|\nabla\sch|\nnm\\
&-\fr{4(m-1)}{3m}(N_1+N_2|H|^2)|\nabla h|^2+2N_2|\sch|^2|H|^2(|h|^2+c_0|H|^2)
+2c_0N_1|\sch|^2|H|^2\Big)\nnm\\
&+\nabla a*(h^3*\nabla h+h*\nabla h)+\nabla^2a*(h^4+h^2).\label{ad2}
\end{align}
The second term on the right hand side of \eqref{ad2} can be estimated by Young's inequality as follows:
\begin{align*}
8N_2|H||\sch||\nabla H||\nabla \sch|
\leq &8N_2|H|\sqrt{m}|\nabla h|^2\sqrt{C_0}|H|^{1-\sigma}\\
\leq &\fr{4(m-1)}{3m}|H|^2|\nabla h|^2+c_1(N_2)|\nabla h|^2
\end{align*}
since $|\sch|^2\leq C_0|H|^{2(1-\sigma)}$; while the last two terms are estimated as
$$2N_2|\sch|^2|H|^2(|h|^2+c_0|H|^2)
+2c_0N_1|\sch|^2|H|^2\leq c_2(N_1,N_2)|\sch|^2(|H|^4+1)$$
since $|h|^2\leq c|H|^2$. Then equation \eqref{6.8} now follows.
\end{proof}

\begin{rmk}\rm
The second inequality of \eqref{ad4} was used \cite{a-b} and \cite{b} with $c_0=2$ but without any proof. In what follows, we would like to provide a proof of it in detail for the convenience of readers. It turns out that, by our argument, $c_0$ can be always taken to be less than $1$ (see \eqref{c0} below) if the constant $c$ satisfies \eqref{c}.

The case that $H=0$ is trivial due to the assumption $|h|^2\leq c|H|^2$. So we only need to consider the case that $H\neq 0$. Then, for positive numbers $\lambda$ and $b<1$, it holds that (\cite{a-b})
\begin{align*}
R_1-\fr1mR_2\leq&|\sch_1|^4+\fr1m|\sch_1|^2|H|^2+4|\sch_1|^2|\sch_-|^2+\fr32|\sch_-|^4\\
=&(|\sch_1|^4+2|\sch_1|^2|\sch_-|^2+|\sch_-|^4)+2b|\sch_1|^2|\sch_-|^2 +\fr12|\sch_-|^4+\fr1m|\sch_1|^2|H|^2\\
&+2(1-b)|\sch_1|^2|\sch_-|^2\\
\leq&|\sch|^4+\lambda b|\sch_1|^4+\left(\fr1\lambda b+\fr12\right)|\sch_-|^4 +\fr1m|\sch_1|^2|H|^2 +2(1-b)|\sch_1|^2|\sch_-|^2,
\end{align*}
where we have used the equality $|\sch|^2=|\sch_1|^2+|\sch_-|^2$ and the Young's inequality $$2|\sch_1|^2|\sch_-|^2\leq\lambda|\sch_1|^4+\fr1\lambda|\sch_-|^4,\quad \text{for any }\lambda>0.$$
By taking
\be\label{lambda}\lambda=\fr1{4b}(1+\sqrt{1+16b^2}),\ee
we have $\lambda b=\fr1\lambda b+\fr12$. It then follows that
\begin{align*}
R_1-\fr1mR_2\leq&|\sch|^4+\lambda b(|\sch_1|^4+2|\sch_1|^2|\sch_-|^2+|\sch_-|^4) +\fr1m|\sch_1|^2|H|^2\\
& +2(1-b-\lambda b)|\sch_1|^2|\sch_-|^2\\
=&|\sch|^4+\lambda b(|\sch_1|^2+|\sch_-|^2)^2 +\fr1m|\sch_1|^2|H|^2 +2(1-b-\lambda b)|\sch_1|^2|\sch_-|^2\\
=&(1+\lambda b)|\sch|^4+\fr1m|\sch_1|^2|H|^2 +2(1-b-\lambda b)|\sch_1|^2|\sch_-|^2.
\end{align*}
Solving $1-b-\lambda b=0$ with \eqref{lambda}, we find that $b=\fr13$ and thus $\lambda b=\fr23$. So we obtain
\begin{align*}
R_1-\fr1mR_2\leq&\fr53|\sch|^4+\fr1m|\sch_1|^2|H|^2\\ \leq&|\sch|^2\left(\fr53|h|^2-\fr5{3m}|H|^2+\fr1m|H|^2\right)\\ =&|\sch|^2\left(\fr53|h|^2-\fr2{3m}|H|^2\right)\\
\leq&\left(\fr53c-\fr2{3m}\right)|\sch|^2|H|^2\\
:=&c_0|\sch|^2|H|^2
\end{align*}
using $|\sch_1|^2\leq |\sch|^2$ and $|h|^2\leq c|H|^2$.

In particular, when $c$ satisfies \eqref{c}, it holds that
\be\label{c0}c_0=\fr53c-\fr2{3m}\leq \begin{cases}\fr79,&m=2;\\
\fr{14}{27},&m=3;\\\fr7{18}&m=4\\
\fr{3m+2}{3m(m-1)}\leq\fr{17}{60}&m\geq5.\end{cases}
\ee
\end{rmk}

Now, for any $N_1,N_2>0$, we define a function
$$f=|\nabla H|^2+(N_1+N_2|H|^2)|\sch|^2.$$
Then, by combining \eqref{ptnabH2} and \eqref{6.8}, we easily obtain

\begin{lem} Suppose $|h|^2\leq c|H|^2$ and $|\sch|^2\leq C_0|H|^{2(1-\sigma)}$ for some $C_0>0$, $0<\sigma<1$. If $N_2$ is large enough, then it holds that
\begin{align}
\pp{}{t}f\leq&a\Big(\Delta f-\fr{4(m-1)}{3m}\left((N_2-1)|H|^2+(N_1-c_1(N_2))\right)|\nabla h|^2\nnm\\
&+c_2(N_1,N_2)|\sch|^2(|H|^4+1)\Big)\nnm\\
&+\nabla a*(h*\nabla h+\nabla h*\nabla^2h+h^3*\nabla h)\nnm\\
&+\nabla^2a*(h^4+h^2+(\nabla h)^2)+\nabla^3a*h*\nabla h.\label{6.9}
\end{align}
\end{lem}

For $\eta>0$, denote $g=f-\eta|H|^4$. Then from \eqref{6.7}, \eqref{6.9} and the fact that $(m+2)|\nabla h|^2\geq 3|\nabla H|^2$ (\cite{a-b}, Proposition 6) follows the next lemma:

\begin{lem}\label{lem6.9}
If $N_1,N_2$ are large enough, $|h|^2\leq c|H|^2$ and and $|\sch|^2\leq C_0|H|^{2(1-\sigma)}$ for some $C_0>0$, $0<\sigma<1$, then it holds that
\begin{align}
\pp{}{t}g\leq&a \Delta g+C_2(N_1,N_2)|\sch|^2(|H|^4+1)-\ul a\fr{4\eta}{m}|H|^6\nnm\\
&+\nabla a*(h*\nabla h+\nabla h*\nabla^2h+h^3*\nabla h)\nnm\\
&+\nabla^2a*(h^4+h^2+(\nabla h)^2)+\nabla^3a*h*\nabla h.\label{6.10}
\end{align}
\end{lem}

\begin{cor}[cf. \cite{a-b}, Theorem 5]\label{cor6.11} Suppose that the function $a=a(t)$, depending only on the parameter $t$. If the initial $F_0$ satisfies $|h|^2\leq c|H|^2$ with $c$ as in \eqref{c} then there exists a $C_0>0$ such that
\be\label{sch}
|\sch|^2\equiv|h|^2-\fr1m|H|^2\leq C_0|H|^{2(1-\sigma)}.
\ee
Furthermore, for each $\eta>0$, there exists a constant $C_\eta$ depending only on $\eta$ and $F_0$ such that
\be\label{6.11}
|\nabla H|^2\leq \eta|H|^4+C_{\eta}.
\ee
\end{cor}

\begin{proof} If $a$ depends only on $t$, then $\nabla a\equiv 0$ along $F_t(M)$ for all $t$. So Corollary \ref{cor6.2} assures that $|h|^2\leq c|H|^2$ holds for any time $t\in [0,T)$, and Proposition \ref{prop6.5} reduces to $\pp{}{t}\int_Mf^p_\sigma dV\leq 0$ for any small $\sigma>0$ and large $p$. It follows that $\int_Mf^p_\sigma dV\leq C$ where $C$ is a constant depending only on $m,p,\ol a,\ul a$ and $F_0$. On the other hand, by Lemma \ref{lem6.4} we also have that
\be\label{ad5}
\pp{}{t}\int_Mf^p_\sigma dV\leq-\fr12\ul a p(p-1)\int_Mf^{p-2}_\sigma|\nabla f_\sigma|^2dV+2\ol ap\sigma\int_M|H|^2f^p_\sigma dV
\ee
which, with the fact that $\|f_\sigma\|_p$ is uniformly bounded, implies in a standard manner (see \cite{hui84} or more directly \cite{b}) that $|\sch|^2\leq C_0|H|^{2(1-\sigma)}$ for some $C_0>0$. From this and \eqref{6.10} we find the estimate
\be\label{6.12}
\pp{}{t}g\leq a\Delta g+C_2(N_1,N_2)(|H|^{6-2\sigma}+|H|^{2(1-\sigma)})-\ul a\fr{4\eta}{m}|H|^6\big).
\ee
Note that $\sigma<1$ is small, and $6-2\sigma$, $2-2\sigma$ are both strictly less than $6$. Therefore we can use the Young's inequality to find
\be\label{6.13}
C_2|H|^{6-2\sigma}\leq\fr{\veps^{p_1}_1}{p_1}|H|^6+\fr1{q_1\veps^{q_1}_1}C^{q_1}_2,\quad C_2|H|^{2(1-\sigma)}\leq\fr{\veps^{p_2}_2 }{p_2}|H|^6+\fr1{q_2\veps^{q_2}_2}C^{q_2}_2
\ee
for small $\veps_1,\veps_2>0$ where
$$
p_1=\fr6{6-2\sigma},\quad p_2=\fr6{2-2\sigma}\text{ and }\fr1{p_i}+\fr1{q_i}=1,\quad i=1,2.
$$
Putting \eqref{6.13} with $\veps_1,\veps_2$ small enough into \eqref{6.12}, we finally obtain the following estimate from:
\be
\pp{}{t}g\leq a\Delta g+c_3
\ee
where $c_3$ is dependent of $\eta$. Now the maximal principle together with the fact that $T<\infty$ concludes that $g\leq C_\eta$ which with the definitions of $f$ and $g$ implies the desired inequality \eqref{6.11}:
$$|\nabla H|^2\leq f=g+\eta|H|^4\leq \eta|H|^4+C_\eta.$$
\end{proof}


\begin{thebibliography}{99}
\bibitem{a--r} L. J. Al\'ias, J. H. de Lira and M. Rigoli, Mean curvature flow solitons in the presence of conformal vector fields, arXiv.org/math.DG/1707.07132v1.
\bibitem{andr} B. Andrews, Contraction of convex hypersurfaces in Riemannian spaces. J. Diff. Geom. v. 39 (1994) 407-431.
\bibitem{a-b} Ben Andrews and Charles Baker, Mean curvature flow of pinched submanifolds to spheres, Journal of Differential Geometry, 2010, 85(3):357-396.
\bibitem{b} Charles Baker, The mean curvature flow of submanifolds of high codimension, arXiv: 1104.4409v1 [math.DG], April 22, 2011.
\bibitem{b-r10} Aleksander Borisenko and Vladimir Rovenski, Gaussian Mean Curvature Flow, J. Evol. Equ. 10 (2010), 413-423; published online February 17, 2010; DOI: 10.1007/s00028-010-0054-2.
\bibitem{b-r14} Aleksander Borisenko and Vladimir Rovenski, Gaussian Mean Curvature Flow for submanifolds in space forms, Geometry and its Applications, Springer Proceedings in Mathematics \& Statistics 72, Springer International Publishing Switzerland 2014; DOI: 10.1007/978-3-319-04675-4-2.
\bibitem{bra} K. Brakke, The motion of a surface by its mean curvature. Prinston, New Jersey: Prinston University Press, 1978.
\bibitem{br} T. Branson, Kato constants in Riemannian geometry, Mathematical Research Letters 7, 245-261 (2000).
\bibitem{c-g-h} David M. J. Calderbank, Paul Gauduchon, Marc Herzlich, Refined Kato inequalities and conformal weights in Riemannian geometry[J]. J. Funct. Anal. 2000, 173(1): 214-255.
\bibitem{c} Bang-Yen Chen, Some pinching and classification theorem for minimal submanifolds, Arch. Math. (Basel), 60(1993),, no. 6, 568-578.
\bibitem{c--n} B. Chow, S. C. Chu, D. Glickenstein, C. Guenter, J. Isenberg, T. Ivey, D. Knopf, P. Lu, F. Luo and L. Ni, The Ricci Flow: Techniques and Applications. Part II: Analytic Aspects, A. M. S., SURV 144, Providence, 2008.
\bibitem{c-k} Bennett Chow and Dan Knopf, The Ricci Flow: An Introduction, Mathematical surveys and monographs, ISSN 0076-5376; v.110, 2004.
\bibitem{dtu} Dennis M. DeTurck, Deforming metrics in the direction of their Ricci
    tensors, Journal of Differential Geometry 18 (1983), no. 1, 157-162.
\bibitem{eck} K. Ecker, Regularity theory for mean curvature flow, Progress in Nonlinear Differential Equations and their Applications, 57. Birkh\"auser Boston, Inc., Boston, MA, 2004.
\bibitem{g-h-l} Sylvestre Gallot, Dominique Hulin and Jacques Lafontaine, Riemannian geometry, 3rd ed., Universitext, Springer-Verlag, Berlin, 2004.
\bibitem{gro} M. Gromov, Isoperimetry of waists and concentration of maps. Geom. Funct. Anal. 13 (2003) 178-215.
\bibitem{g-l-w} Shunzi Guo, Guanghan Li and Chuanxi Wu, Mean Curvature Flow with a Forcing Field in Direction of the Position Vector of Submanifolds, Acta Math. Sinica (Chinese Series), 2014, Vol. 57 (1): 139-150.
\bibitem{ha82} Richard S. Hamilton, Three manifolds with positive Ricci curvature, J. Differential Geometry, 17(1982) 255-306.
\bibitem{har} Philip Hartman, Systems of Total Differential Equations and Liouville's theorem on Conformal Mapping, American Journal of Mathematics 69(2): 1947, 329¨C332.
\bibitem{h-s} D. Hokman, J. Spruck. Sobolev and isoperimetric inequalities for Riemannian submanifolds, Comm. Pure Appl. Math. 1974, 27: 715-727.
\bibitem{hui84} Gerhard Huisken, Flow by mean curvature of convex surfaces into spheres, J. Differential Geometry, 20(1984), 237-266.
\bibitem{hui87} Gerhard Huisken, Deforming Hypersurfaces of the Sphere by Their Mean Curvature, Math. Z. 195(1987), 205-219.
\bibitem{hui93} G. Huisken, Local and Global Behaviour of Hypersurfaces Moving by Mean Curvature, Procedings of Symposia in Pure Mathematics Volume 54 (1993), Part I, 175-191.
\bibitem{j-x03} H. Y. Jian and X. W. Xu, The vortex dynamics of a Ginzburg-Landau system under pinning effect. Science in China Ser A, 46: 488-498 (2003).
\bibitem{j-l06} H. Y. Jian and Y. N. Liu, Ginzburg-Landau vortex and mean curvature flow with external force field. Acta Math Sin Engl Ser 22: 1831-1842 (2006)
\bibitem{k-s} David Kinderlehrer and Guido Stampacchia, An introduction to variational inequalities and their applications, Classics in Applied Mathematics, vol. 31, Society for Industrial and Applied Mathematics (SIAM), Philadelphia, PA, 2000. Reprint of the 1980 original.
\bibitem{l-s} G. H. Li and Isabel Salavessa, Forced convex mean curvature flow in Euclidean spaces, manuscripta mathematica, July 2008, 126, 3, 333-351.
\bibitem{lhz} H. Z. Li, Lecture Notes on Mean Curvature Flow, 2, 2016, Japan.
\bibitem{l-j07} Y. N. Liu and H.-Y. Jian, Evolution of hypersurfaces by the mean curvature minus an external force field, Science in China Series A: Mathematics, February 2007, 50, 2, 231?39
\bibitem{l-j08} Y. N. Liu and H.-Y. Jian, Long-time existence of mean curvature flow with external force fields, Pacific J. Math., 234, 2, 2008, 311-324.
\bibitem{l--z12} K. F. Liu, H. W. Xu and E. T. Zhao, Mean curvature flow of higher codimension in Riemannian manifolds, arXiv.org/math.DG/1204.0107v1.
\bibitem{l--z11} K. F. Liu, H. W. Xu, Fei Ye and E. T. Zhao, Mean curvature flow of higher codimension in hyperbolic spaces, arXiv.org/math.DG/1105.5686v1.
\bibitem{l--z} K. F. Liu, H. W. Xu, Fei Ye and E. T. Zhao, The extension and convergence of mean curvature flow in higher codimension, arXiv.org/math.DG/1104.0971.
\bibitem{m-s} J. H. Michael and L. M. Simon, Sobolev and mean value inequalities on generalized submanifolds of $R^n$, Comm. Pure Appl. Math. 26(1973), 316-379.
\bibitem{m} G. Monge, Application de l'analyse \`a la g\'eom\'etrie, Bachelier (1850), 609-616.
\bibitem{mntl} S. Montiel, Unicity of constant mean curvature hypersurfaces in some Riemannian manifolds, Indiana Univ. Math. J. 48 (1999), no. 2, 711-48.
\bibitem{mor} F. Morgan, Manifolds with density, Notices Am. Math. Soc. 52, (2005), 853-858
\bibitem{p} Peter Petersen, Riemannian Geometry, second edition, Graduate Texts in Mathematics, Vol 171, Springer, New York, 2006.
\bibitem{r--m} C. Rosales, A. Ca\~nete, V. Bayle, and F. Morgan, On the isoperimetric problem in Euclidean space with density, Calc. Var. 31-46
\bibitem{s-s} O. C. Schn\"urer and K. Smoczyk, Evolution of hypersurfaces in central force fields, J. Reine Angew. Math. 550 (2002), 77-95
\bibitem{smo01} K. Smoczyk, A relation between Mean Curvature Flow Solitons and Minimal Submanifolds, Math. Nachr., 229 (2001), 175-186.
\bibitem{smo11} K. Smoczyk, Mean curvature flow in higher codimension\,---\,Introduction and survey, arXiv: 1104.3222v2 [math.DG], April, 2011.
\bibitem{sp} Michael Spivak, A comprehensive introduction to differential geometry. Vol. IV, 2nd ed., Publish or Perish Inc., Wilmington, Del., 1979.
\bibitem{w} M.-T. Wang, Lectures on mean curvature flow in higher codimensions, Handbook of geometric analysis. No. 1, 525?43, Adv. Lect. Math. (ALM), 7, Int. Press, Somerville, MA, 2008 (see also arXiv.org/math.DG/1104.3354 v1).
\bibitem{x-g} H. W. Xu and J. R. Gu, An optimal differentiable sphere theorem for complete manifolds, Math. Res. Lett. 17 (2010), 1111-1124.
\bibitem{yan} K. Yano, Integral formulas in Riemannian Geometry, Marcel Dekker, New York, 1970.
\bibitem{zhu} Xi-Ping Zhu, Lectures on mean curvature flows, AMS/IP, Providence, 2002.
\end{thebibliography}
\end{document}